\documentclass[12pt,a4paper]{amsart}
\usepackage[top=35mm, bottom=40mm, left=30mm, right=30mm]{geometry}
\usepackage[colorlinks=true,citecolor=blue]{hyperref}
\usepackage{mathptmx}
\usepackage{eucal}
\usepackage{graphicx}
\usepackage{mathrsfs}
\usepackage{amssymb}
\usepackage{amsmath}
\usepackage{amsthm}
\usepackage{amscd}
\usepackage{xcolor}
\usepackage{verbatim}

\newtheorem{thmx}{Theorem}

\newtheorem{thm}{Theorem}[section]
\newtheorem{cor}[thm]{Corollary}
\newtheorem{lem}[thm]{Lemma}
\newtheorem{prop}[thm]{Proposition}

\theoremstyle{definition}
\newtheorem{defn}[thm]{Definition}

\newtheorem{exmp}[thm]{Example}

\numberwithin{equation}{section}

\DeclareMathOperator{\diam}{diam}

\DeclareMathOperator{\supp}{supp}

\newcommand{\N}{\mathbb{N}}

\newcommand{\F}{\mathcal{F}}

\newcommand{\Z}{\mathbb{Z}}
\newcommand{\T}{\mathbb{T}}
\newcommand{\R}{\mathbb{R}}

\newcommand{\ep}{\epsilon}
\newcommand{\lra}{\longrightarrow}
\newcommand{\ra}{\rightarrow}

\def \d {\delta}
\def \RP {{\bf RP}}

\numberwithin{equation}{section}

\input xy
\xyoption{all}

\begin{document}
\title{Sensitivity, proximal extension and higher order almost automorphy}
\author{Xiangdong Ye and Tao Yu}
\address[X. Ye, T. Yu]{Wu Wen-Tsun Key Laboratory of Mathematics, USTC, Chinese Academy of Sciences,
Department of Mathematics, University of Science and Technology of China,
Hefei, Anhui, 230026, P.R. China}
\email{yexd@ustc.edu.cn}
\email{ytnuo@mail.ustc.edu.cn}
\thanks{The authors were supported by NNSF of China (11371339, 11431012, 11571335).}

%\date{Sept. 1, 2014}
%\date{April 1, 2015}
%\date{June 10, 2015}

\begin{abstract}
Let $(X,T)$ be a topological dynamical system, and $\F$ be a family of subsets of $\Z_+$. $(X,T)$ is strongly $\F$-sensitive,
if there is $\delta>0$ such that for each non-empty open subset $U$,
there are $x,y\in U$ with $\{n\in\Z_+: d(T^nx,T^ny)>\delta\}\in\F$.
Let $\F_t$ (resp. $\F_{ip}$, $\F_{fip}$) be consisting of thick sets
(resp. IP-sets, subsets containing arbitrarily long finite IP-sets).

The following Auslander-Yorke's type dichotomy theorems are obtained:
(1) a minimal system is either strongly $\F_{fip}$-sensitive
or an almost one-to-one extension of its $\infty$-step nilfactor.
(2) a minimal system is either strongly $\F_{ip}$-sensitive
or an almost one-to-one extension of its maximal distal factor.
(3) a minimal system is either strongly $\F_{t}$-sensitive
or a proximal extension of its maximal distal factor.

%Results on other families are also obtained.
%(3) a minimal system is either strongly  thick sensitive
%or a proximal extension of its maximal equicontinuous factor;
\end{abstract}

\maketitle

\section{Introduction}

Throughout this paper $(X,T)$ denotes a \emph{topological dynamical system} (t.d.s. for short),
where $X$ is a compact metric space, %with a metric $d$
and $T: X\ra X$ is continuous and surjective.
In this section, we first discuss the motivations of our research and then state
the main results of the article.

The notion of sensitivity (sensitive dependence on initial conditions) was first
used by Ruelle \cite{Ruelle1977}. It is in the kernel of the definition of Devaney's chaos.
 According to  Auslander and Yorke \cite{Auslander1980}
a t.d.s. $(X,T)$ is called \emph{sensitive}
if there exists $\delta>0$ such that for every $x\in X$ and every neighborhood $U_x$ of $x$,
there exist $y\in U_x$ and
$n\in\N$ with $d(T^nx,T^ny)>\delta$.
%, or equivalently if there exists $\delta>0$ such that for every open
%and non-empty subset $U$, there exist $x,y\in U$ and $n\in\N$ with $d(T^nx,T^ny)>\delta$.
For a t.d.s. $(X,T)$, $\delta>0$ and an opene (open and non-empty) subset $U\subset X$, put
$$N(\delta,U)=\{n\in\N: \exists x,y \in U \ \text{with} \ d(T^nx,T^ny)>\delta\}=\{n\in\N: \diam(T^n(U))>\delta\}.$$
Then it is easy to see that $(X,T)$ is sensitive if and only if there exists $\delta>0$ such that
$N(\delta,U)\not = \emptyset$ for each opene subset $U$.
A t.d.s. $(X,T)$ is called \emph{equicontinuous} if for every $\varepsilon>0$
there is a $\delta>0$ such that
whenever $x,y\in X$ with $d(x,y)<\delta$, then $d(T^nx,T^ny)<\varepsilon$ for $n\in \N$.
Auslander and Yorke \cite{Auslander1980} proved the following dichotomy theorem:
a minimal system is either equicontinuous or sensitive.
A similar result obtained by Glasner and Weiss \cite{Glasner1993} states that:
a transitive system is either almost equicontinuous or sensitive.

%Let $(X,T)$ be a t.d.s. and $\mu\in M(X,T)$. We say that $(X,T)$ is $\mu$-sensitive if
%there exists $\delta>0$ such that for any $A\in \mathcal{B}$ with positive measure,
%there are $x,y\in A$ and $n\in\N$, such that $d(T^nx,T^ny)>\delta$.

%Let $(X,T)$ be a t.d.s. with metric $d$ and $K\subseteq X$. We say
%that $K$ is {\it $T$-equicontinuous} if for any $\epsilon>0$ there
%is $\delta=\delta(\epsilon)>0$ such that when $x,y\in K$ with
%$d(x,y)<\delta$ then $d(T^nx,T^ny)\le \epsilon$ for all $n\ge 0$. If
%$X$ itself is $T$-equicontinuous, then $T$ or $(X,T)$ is called {\it
%equicontinuous}.
There are several attempts to generalize the notion of sensitivity.
Akin and Kolyada \cite{Akin2003} introduced the notion of Li-Yorke sensitivity, combining the
two well known notions (sensitivity and Li-Yorke chaos) together. The study of sensitivity related to families of non-negative
integers was initiated by Moothathu in \cite{Moothathu2007}.
%Some another way to measure the sensitivity of a system, by checking how large is
%the set of nonnegative integers for which the sensitivity still happens,
Let $\F$ be a family. Recall that according to \cite{Moothathu2007} $(X,T)$ is $\F$-{\it sensitive}
if there is $\delta>0$ such that for any opene subset $U$, $N(\delta,U)\in \F$.
%there is $k\in \N$ with $\diam(T^k(U_i))>\delta$ for $1\le i\le n$.
$\F$-sensitivity for some families were discussed in \cite{Moothathu2007, Felipe14, Shi, Wang, HKZ2014, Li-Ye}.
It is known that
for a minimal system $\{thick\}$-sensitivity
is equivalent to $\{thickly\ syndetic\}$-sensitivity \cite{Wang}.
Very recently, a striking result obtained by
Huang, Kolyada and Zhang \cite[Theorem 3.1]{HKZ2014}\label{HKZ14} states that:
a minimal system is either
$\{thick\}$-sensitive or an almost one-to-one extension of its maximal equicontinuous factor.

\medskip
It is clear that when $(X,T)$ is $\F$-sensitive, then
%there is $\delta>0$ such that $N(\delta,U)\in \F$ for any non-empty open subset $U$.
for $n\in N(\delta,U)$, there are $x_n,y_n \in U$ such that $d(T^nx_n,T^ny_n)>\delta$.
If we require all $x_n$ (resp. $y_n$) are equal,
then it leads the notion of {\it strong $\F$-sensitivity}
which will be studied in detail in the paper. Recall that $(X,T)$ is strongly $\F$-sensitive,
if there is $\delta>0$ such that for each non-empty open subset $U$,
there are $x,y\in U$ with $\{n\in\Z_+: d(T^nx,T^ny)>\delta\}\in\F$, where $\F$ is a family of subsets of $\Z_+$.
We remark that some notions of sensitivity similar with the strong sensitivity
were studies in \cite{LTY15, Felipe14},
which appear naturally when studying mean equicontinuity. It was shown that a minimal system is either
mean-sensitive, or mean equicontinuous.

When investigating strong sensitivity we find that for some families $\F$ the requirement of
all $x_n$ or $y_n$ being equal is too strong.
So in this paper we also introduce a notion of sensitivity related
to a family $\F$, called {\it block $\F$-sensitivity}.
Roughly speaking, in this definition we require
$x_n$ (resp. $y_n$) are equal for a sequence of arbitrarily  long finite segments from the family $\F$.
For example,
a t.d.s. $(X,T)$ is called \emph{block \{thick\}-sensitive} (resp. \emph{block \{IP\}-sensitive})
if there is $\delta>0$
such that for each $x\in X$, every neighborhood $U_x$ of $x$ and
$l\in\N$ there are $y_l\in U_x$ with
$\{n\in\Z_+:d(T^nx,T^ny_l)>\delta\}$ containing $\{m+1,\ldots,m+l\}$ for some $m=m(l)\in \N$
(resp. a finite IP-set of length at least $l$.)
%It is easy to see that block $\{IP\}$-sensitivity is the same as strong $\F_{fip}$-sensitivity.
Thus
$$strong\ \F\!-\!sensitivity \subset  block\ \F-sensitivity \subset \F-sensitivity. $$

%But points $x_n,y_n$ may change as $n$ changes. If the system $(X,T)$ is $\F$ sensitive
%and $x_n,y_n$ does not change as $n$ varies through $N(\delta,U)$,
%the system is called \emph{strongly $\F$ sensitive system}.
%If $(X,T)$ is sensitive, there exists $\delta>0$ such that for every $x\in X$
%and every neighborhood $U_x$ of $x$, there exist
%$y\in U_x$ and %$n\in\N$ with $d(T^nx,T^ny)>\delta$.  We give another concept
%about sensitive on family,
%which is stronger than the concept of Moothathu.

%\begin{defn}
%A dynamical system $(X,T)$ is called \emph{inter thick sensitive} if there is $\delta>0$
%such that for each $x\in X$, every neighborhood $U_x$ of $x$
%and $l\in\N$ there are $y_l\in U_x$ with
%$\{n\in\Z_+:d(T^nx,T^ny_l)>\delta\}$ containing $\{m+1,\ldots,m+l\}$ for some $m\in \N$.
%\end{defn}

%\begin{defn}
%A dynamical system $(X,T)$ is called \emph{inter finite IP sensitive} if there is $\delta>0$
%such that for each $x\in X$, every neighborhood $U_x$ of $x$ and $l\in \N$ there is $y_l\in U$
%such that $\{n: d(T^nx,T^ny_l)>\delta\}$ contains a finite IP of length at leat $l$.
%\end{defn}

%Inspiring by Huang's theorem \cite[Theorem 3.1]{HKZ2014}\label{HKZ14}, we have a question:
%if the minimal system  is not strongly thick (strongly IP, inter thick, inter finite IP) sensitive,
%the extension between the system and its maximal equicontinuous factor
%(maximal distal factor, maximal d-step nilfactor, maximal $\infty$-step nilfactor)

%The main results of this paper are affirmed answer to the previous question
%and the idea of these proof is through the whole paper.

In this paper first we investigate $\F$-sensitivity to warm up.
%, aiming to extend some known results.
Then we study block $\F$-sensitivity and some related strong $\F$-sensitivity notions for some families.
Finally we will focus on strong $\F$-sensitivity.
Note that for a minimal system we use $X_{eq}$, $ X_{\infty}$ and $X_D$ to denote the maximal equicontinuous factor,
the maximal $\infty$-step nilfactor and the maximal distal factor of $X$ respectively (for the definitions see Section \ref{section2}).
It is very interesting that for some well known families strong sensitivity for the family is closely related to
other well known dynamical properties.

The main results of the paper are:
%Let $\F_{t}$ be the family of all thick subsets.

%\medskip

\begin{thmx}
Let $(X,T)$ be a minimal system. Then the following conditions are equivalent:
\begin{enumerate}
  \item $(X,T)$ is block $\F_t$-sensitive;
  \item $\pi\colon X\to X_{eq}$ is not proximal.
\end{enumerate}
\end{thmx}
\medskip

\begin{thmx}
Let $(X,T)$ be an invertible minimal system.
Then the following statements are equivalent.
\begin{enumerate}
\item  $(X,T)$ is  strongly $\F_{fip}$-sensitive;
\item  $(X,T)$ is  block $\F_{ip}$-sensitive;
\item   $\pi: X\lra X_{\infty}$ is not almost one-to-one.
\end{enumerate}
\end{thmx}

\newpage

%{\color{blue} If $\pi\colon X\to X_\infty$ is not proximal, then it not almost one-to-one. This implies that $(X,T)$ is strongly $\F_{fip}$-sensitive
%and hence strongly $\F_{f\Delta}$-sensitive.}

\begin{thmx}
Let $(X,T)$ be a minimal system. Then the following conditions are equivalent:
\begin{enumerate}
  \item $(X,T)$ is strongly $\F_{ip}$-sensitive;
  \item $\pi\colon X\to X_D$ is not almost one-to-one. %  where $(X_D,T)$ is the maximal distal factor of $(X,T)$.
\end{enumerate}
\end{thmx}

%we can show that if a minimal t.d.s. $(X,T)$ is not strongly $\F_t$-sensitive,
%then $(X,T)$ is PI (Theorem \ref{1009-4}) and a proximal factor
%of a proximal extension over a distal system is
%not strongly $\F_t$-sensitive (Corollary \ref{1010-11}). Moreover, we have
%\newpage

\begin{thmx} Let $(X,T)$ be a minimal system.
Then the following conditions are equivalent:
\begin{enumerate}
  \item $(X,T)$ is strongly $\F_{t}$-sensitive;
  \item $\pi\colon X\to X_D$ is not proximal. % where $(X_D,T)$ is the maximal distal factor of $(X,T)$.
\end{enumerate}
\end{thmx}

% and have the structure described in Lemma~\ref{structure}. Then $(X,T)$ is
%not strongly $\F_t$-sensitive if and only if $(X,T)$ is PI and each finite proximal
%and isomorphic tower
%$Z_1\overset{\theta_1}{\leftarrow}Y_1\overset{\rho_2}{\leftarrow}Z_2 \overset{\theta_2}{\leftarrow}
%Y_2\overset{\rho_3}{\leftarrow}Z_3\overset{\theta_3}{\leftarrow}  \ldots \overset{\rho_n}{\leftarrow}
% Z_n\overset{\theta_n}{\leftarrow}Y_n\overset{\rho_{n+1}}{\leftarrow}Z_{n+1}$
%in the structure of $(X,T)$ is
%not strongly $\F_t$-sensitive, where $Z_1$ is equicontinuous,
%$\theta_i$ is proximal and $\rho_i$ is isometric.
%Note that we can show the Morse system is strongly $\F_t$-sensitive,
%and so a natural conjecture is that
%a minimal t.d.s. $(X,T)$ is not strongly $\F_t$-sensitive if and only
%if it is a proximal factor of a proximal extension over a distal system.

%Moreover, strong $\F_t$-sensitivity can be read through the finite
%equicontinuous-proximal towers in the structure
%of $X$ (Theorem~\ref{2015-4-1}).

From Theorem B it is natural to ask if we can find some
family $\F$ such that strong $\F$-sensitivity is related to a
$d$-step almost automorphy (see Section \ref{sec-5.2} for the definitions of the families appeared below), $d\in\N$.
This leads us to study strong $\F_{Poin_d}$-sensitivity (where $\F_{Poin_d}$ is the family of
all $d$-step Poincar\'e sequences) for $d\in\N$.
We show that if a minimal t.d.s. $(X,T)$ is strongly $\F_{Poin_d}$-sensitive,
then $\pi: X\lra X_{d}$ is not an almost one-to-one extension (Theorem \ref{1011-1}),
where $X_d$ is the maximal $d$-step nilfactor of $X$.
Examples show that the converse statement does not hold (see Example \ref{last-e}).
It is an interesting open question to find a family $\F$ such that for any minimal system
$(X,T)$, $(X,T)$ is strongly $\F$-sensitive if and only $\pi:X\lra X_\infty$ is not proximal.

For a minimal system, Table 1 gives the details of results obtained in the paper (the
results related to sensitivity are essentially obtained in \cite{HKZ2014}).

\begin{table}[!h]
\caption{Relationships} \vspace*{1.5pt}
\begin{center}
\begin{tabular}{| p{2.0cm}| p{4cm} | p{4.0cm} | p{2.0cm} |p{2.0cm}|}
\hline \multicolumn{1}{|c|} {}  & \multicolumn{1}{|c|}
{not strongly sensitive} & \multicolumn{1}{|c|} {not block sensitive} & \multicolumn{1}{|c|}
{not sensitive}   \\
\hline
$\F_t$ & proximal extension of the maximal distal factor & proximal extension of maximal equi. factor& almost  automorphy  \\
\hline
$\F_{ip}$ & {almost 1-1 extension of maximal distal factor} & $\infty$-step almost  automorphy & almost  automorphy \\
\hline
$\F_{fip}$ & $\infty$-step almost  automorphy& $\infty$-step almost  automorphy &almost  automorphy \\
\hline
%$\F_{Poin_d}$ & $\subset$ $\infty$-step almost  automorphy, $\supset$ $d$-step almost  automorphy  & no definition & almost  automorphy \\
%\hline

%Banach & regular AA?? & no definition & AA \\
%\hline

\end{tabular}
\end{center}
\end{table}

%if and only if $\pi: X\lra X_{d}$ is not almost one-to-one extension.}

\medskip
We remark that when defining strong sensitivity, except for the definition given
before one may define strong $\F$-sensitivity
as follows: if there is $\delta>0$ such that for each $x\in X$ and each neighborhood $U$ of $x$,
there is $y\in U$ with $\{n\in\Z_+: d(T^nx,T^ny)>\delta\}\in\F$.
It is easy to see that the two definitions coincide when $\F$ has the Ramsey property.
We also remark that since any sensitive minimal system is strongly $\{syndetic\}$-sensitive \cite{Moothathu2007},
we know that if a family $\F$ contains the set of all syndetic subsets then for a minimal system strong $\F$-sensitivity
is equivalent to sensitivity.
This fact restricts the families when we consider strong $\F$-sensitivity and try to obtain new results,
and also explains the reason why we choose $\F_t$, $\F_{ip}$ and $\F_{fip}$ et al to consider
strong $\F$-sensitivity in this paper.

We also remark that for a transitive system, we may investigate the same problem. As the restriction
of the length of the paper we leave this study to readers.

The paper is organized as follows: In Section 2, we recall some definitions and  some related theorems.
In Section 3, we discuss sensitivity. In Section 4, we study block sensitivity and some related notions of
strong sensitivity, and prove Theorem A, Theorem B and Theorem C. In Section 5, we
investigate strong sensitivity and show Theorem D.

\medskip
\noindent {\bf Acknowledgments.}
The authors would like to thank Wen Huang, Song Shao for very useful discussions; and Jian Li
and Guohua Zhang for very careful reading.
%{The authors would also like to thank the anonymous referee for his/her
%helpful suggestions concerning this paper.}

\section{Preliminaries}\label{section2}
In this section we will recall some basic notions and theorems we need in the following sections.

\subsection{Topological dynamical systems}
In the article, sets of integers, nonnegative integers and
natural numbers are denoted by $\mathbb{Z},\ \mathbb{Z}_+$ and $\mathbb{N}$ respectively.
By a topological dynamical system we mean a pair $(X, T)$,
where $X$ is a compact metric space with a metric $d$ and $T : X\rightarrow X$ is continuous and surjective.
A non-vacuous closed
invariant subset $Y\subseteq X$ defines naturally a subsystem $(Y, T)$ of $(X, T)$.
A system $(X,T)$ is called {\em minimal} if it contains no proper subsystem.
Each point belonging to some minimal subsystem of $(X,T)$ is called a {\em minimal point}.
The \emph{orbit} of a point $x\in X$ is the set $Orb(x,T)=\{T^nx:\ n\in\Z_+\}$.

For $x\in X$ and $U,V\subset X$, put
$$N(x,U)=\{n\in\Z_+: T^nx\in U\}\ \text{and}\ N(U,V)=\{n\in\Z_+: U\cap T^{-n}V\neq\emptyset\}.$$
Recall that a dynamical system $(X,T)$ is called {\em topologically transitive}
(or just {\em transitive}) if for every two opene subsets $U,V$ of $X$
the set $N(U,V)$ is infinite.  Any point with dense orbit is called a \emph{transitive point}.
Denote the set of all transitive points by $Trans(X,T)$.
It is well known that for a transitive system, $Trans(X,T)$ is a dense $G_\delta$ subset of $X$.

Let $M(X)$ be the set of all Borel probability measures on $X$.
We are interested in those members of $M(X)$ that are invariant measures for $T$, denote by $M(X,T)$.
This set consists of all
$\mu\in M(X)$ making $T$ a measure-preserving transformation of $(X,\mathcal{B}(X),\mu)$,
where $\mathcal{B}(X)$ is the Borel $\sigma$-algebra of $X$.
By the Krylov-Bogolyubov Theorem, $M(X,T)$ is nonempty.
The \emph{support} of a measure $\mu\in M(X)$, denoted by $\supp(\mu)$,
is the smallest closed subset $C$ of $X$ such that $\mu(C)=1$.
We say that a measure has \emph{full support} or is \emph{fully supported} if $\supp(\mu)=X$.
If $(X,T)$ is a minimal system, every $T$-invarant measure has full support.

\subsection{Distal, proximal, regionally proximal}
Let $(X,T)$ and $(Y,S)$ be two dynamical systems. If there is a continuous surjection
$\pi: X \to Y$ with $\pi\circ T = S\circ \pi$,
then we say that $\pi$ is a \emph{factor map}, the system
$(Y,S)$ is a \emph{factor} of $(X,T)$ or $(X, T)$ is an \emph{extension} of $(Y,S)$.
If $\pi$ is a homeomorphism, then we say that $\pi$ is a \emph{conjugacy} and
dynamical systems $(X,T)$ and $(Y,S)$ are \emph{conjugate}. Conjugate dynamical systems can
be considered the same from the dynamical point of view.

Let $(X,T)$ be a dynamical system. A pair  $(x_1,x_2)\in X\times X$ is said to be \emph{proximal}
if for any $\ep>0$, there exists a positive integer $n$ such that $d(T^nx_1,T^nx_2)<\ep$.
Let $P(X,T)$ denote the collection of all proximal pairs in $(X,T)$, $P$ is a reflexive symmetric $T$
invariant relation, but is in general not transitive or closed. If $(x,y)$ is not
proximal, it is said to be a {\it distal pair}. A system $(X,T)$ is called {\it distal} if any pair
of distinct points in $(X,T)$ is a distal pair.

Recall that the
{\it regionally proximal relation} $Q(X, T)$ is the set of all points $(x_1, x_2)\in X\times X$ such that
for each $\ep> 0$ and each open neighborhood $U_i$ of $x_i$, $i = 1, 2$,
there are $x_i'\in U_i, i = 1, 2$,
and $n\in \N$ with $d(T^n(x_1'), T^n(x_2'))<\ep$.
Note that $Q(X)$ is a reflexive symmetric $T$
invariant closed relation, but is in general not transitive.
However for each minimal system $(X,T)$, $Q(X)$ is a closed invariant equivalence relation.

Every topological dynamical system $(X,T)$ has a maximal distal factor $(X_D,T)$
and a maximal equicontinuous factor $(X_{eq},T)$. That is, $(X_D,T)$ is distal
and every distal factor of $(X,T)$ is
a factor of $(X_D, T)$. $(X_{eq}, T)$ has the corresponding property for equicontinuous factors.
Thus there are
closed $T$-invariant equivalence relations $S_D$ and $S_{eq}$ such that $X/S_D=X_D$
and $X/S_{eq}=X_{eq}$. $S_D$ is the smallest closed $T$-invariant
equivalence relation containing $P(X)$, and $X_{eq}$ is the smallest closed $T$ invariant
equivalence relation containing $Q(X)$.

An extension $\phi:(X, T) \rightarrow (Y, S)$ is {\it proximal} if $R_{\phi}\subset P(X, T)$ and is distal if
$R_{\phi}\cap P(X, T) = \Delta_{X}$, where $R_\phi=\{(x,y)\in X^2: \phi(x)=\phi(y)\}$.
Observe that when $Y$ is trivial (reduced to one point) the map
$\phi$ is distal if and only if $(X, T)$ is distal.
An extension $\phi : (X, T) \rightarrow (Y, T)$
is almost one-to-one if the $G_{\delta}$ set $X_0 = \{x \in X : \phi^{-1}(\phi(x)) = {x}\}$ is dense.

\subsection{Nilmanifolds and nilsystems}

Let $G$ be a group. For $g, h\in G$, we write $[g, h] =ghg^{-1}h^{-1}$
for the commutator of $g$ and $h$ and we write
$[A,B]$ for the subgroup spanned by $\{[a, b] : a \in A, b\in B\}$.
The commutator subgroups $G_j$, $j\ge 1$, are defined inductively by
setting $G_1 = G$ and $G_{j+1} = [G_j ,G]$. Let $k \ge 1$ be an
integer. We say that $G$ is {\em $k$-step nilpotent} if $G_{k+1}$ is
the trivial subgroup.

\medskip

Let $G$ be a $k$-step nilpotent Lie group and $\Gamma$ a discrete
cocompact subgroup of $G$. The compact manifold $X = G/\Gamma$ is
called a {\em $k$-step nilmanifold}. The group $G$ acts on $X$ by
left translations and we write this action as $(g, x)\mapsto gx$.
The Haar measure $\mu$ of $X$ is the unique probability measure on
$X$ invariant under this action. Let $\tau\in G$ and $T$ be the
transformation $x\mapsto \tau x$ of $X$. Then $(X, T, \mu)$ is
called a {\em basic $k$-step nilsystem}. When the measure is not
needed for results, we omit it and write that $(X, T)$ is a basic
$k$-step nilsystem.

\medskip

We also make use of inverse limits of nilsystems and so we recall the definition of an inverse
limit of systems (restricting ourselves to the case of sequential inverse limits).
If $(X_i,T_i)_{i\in \N}$ are systems with $diam(X_i)\le M<\infty$ and $\phi_i:X_{i+1}\rightarrow X_i$
are factor maps, the {\em inverse limit} of
the systems is defined to be the compact subset of $\prod_{i\in\N}X_i$ given by
$\{ (x_i)_{i\in \N }: \phi_i(x_{i+1}) = x_i,i\in\N\}$, which is denoted by
$\displaystyle \lim_{\longleftarrow}\{X_i\}_{i\in \N}$. It is a compact metric space endowed with the distance
$\rho(x, y) =\sum_{i\in \N} 1/2^i d_i(x_i, y_i )$.
We note that the maps $\{T_i\}$ induce a
transformation $T$ on the inverse limit.
Let $(X_i,T_i)=(X,T)$ and $\phi_i=T$,
then the inverse limit of systems $(\widetilde{X},\widetilde{T})$ is called
the natural extension of $(X,T)$.

\medskip

If  $(X,T)$ is an inverse limit of basic $(d - 1)$-step minimal nilsystems.
$(X,T)$ is called a {\em $(d-1)$-step nilsystem} or a {\em system of order $(d-1)$}.

\subsection{Regionally proximal relation of order \texorpdfstring{$d$}{d},
\texorpdfstring{$\RP^{[d]}$}{RPd}}
Let $(X, T)$ be a t.d.s. and let $d\ge 1$ be an integer. A pair $(x, y) \in X\times X$ is said to be
{\em regionally proximal of order $d$} if for any $\d> 0$, there exist $x', y'\in X$ and a
vector ${\bf n} = (n_1,\ldots, n_d)\in\Z^d$ such that $\rho(x, x') < \d, \rho(y, y') <\d$, and
$$ \rho(T^{{\bf n}\cdot \ep}x', T^{{\bf n}\cdot \ep}y') < \d\ \text{for
any $\ep\in \{0,1\}^d$, $\ep\not=(0,\ldots,0)$},
$$ where ${\bf n}\cdot \ep = \sum_{i=1}^d \ep_in_i$. The set of
regionally proximal pairs of order $d$ is denoted by $\RP^{[d]}(X)$,
is called {\em the regionally proximal relation of order $d$}.

This notion was first introduced by Host-Kra-Maass in \cite{HKM10}.
It is clear that
\begin{equation}\label{easy}
 P(X)\subseteq \ldots \subseteq \RP^{[d+1]}\subseteq
\RP^{[d]}\subseteq \ldots \subseteq \RP^{[2]}\subseteq \RP^{[1]}=Q(X).
\end{equation}

It was shown \cite{HKM10, SY2012} that for each minimal system $(X,T)$,
$\RP^{[d]}(X)$ is a closed invariant equivalence relation
for any $d\in \N$. When $d=1$, $\RP^{[d]}(X)$ is nothing
but the classical regionally proximal relation which determines the
maximal equicontinuous factor for any minimal system. We remark that
recently Glasner-Gutman-Ye \cite{GGY16} define a new regionally proximal
relation of order $d$ for any group $G$ (coinciding with the previous definition
when $G$ is abelian) and show that it is an equivalence relation for any minimal
system $(X,G)$.

Now we state a proposition from \cite{HKM10, SY2012}
which we need in the sequel.

\begin{prop}\label{nil-factor-d}
Let $(X,T)$ be  minimal systems and $d\in \N$. Then the following statements are equivalent:
\begin{enumerate}
  \item $(X,T)$ is a $d$-step nilsystem;
  \item $\RP^{[d]}(X)=\Delta_X$.
\end{enumerate}
\end{prop}

Let $\RP^{[\infty]}(X)=\underset{d=1}{\overset{\infty}{\cap}}\RP^{[d]}(X)$, then $\RP^{[\infty]}(X)$
is a closed invariant equivalence relation.

\begin{defn}
A minimal system $(X, T)$ is an $\infty$-step nilsystem or a system of order $\infty$,
if the equivalence relation $\RP^{[\infty]}$ is trivial, i.e. coincides with the diagonal.
\end{defn}

The following proposition was proved in \cite{5p}.
\begin{prop} A minimal system is an $\infty$-step nilsystem if and only if it is an
inverse limit of minimal nilsystems.
\end{prop}

Let $(X,T)$ be a t.d.s. and $d\in \N$,
put $X_d=X/ \RP^{[d]}(X)$ and $X_{\infty}=X/\RP^{[{\infty}]}(X)$.

\begin{defn}\index{$d$-step almost automorphic point}
Let $(X,T)$ be a minimal system and $d\in \N\cup\{\infty\}$.  A point $x\in
X$ is called a {\em $d$-step almost automorphic} point (or $d$-step
AA point for short) if $\RP^{[d]}(X)[x]=\{x\}$.
%, where$Y=\overline{\{T^nx:n\in \mathbb{Z}\}}$ and $\RP^{[d]}(Y)[x]=\{y\in
%Y: (x,y)\in \RP^{[d]}(Y)\}$.

A minimal system $(X,T)$ is called {\em $d$-step almost automorphic} ($d$-step AA for short)
if it has a $d$-step almost automorphic point.
\end{defn}

$d$-step almost automorphic systems were studied systematically in \cite{WSY2015},
in particular we have

\begin{prop}\cite[Theorem 8.13]{WSY2015}\label{thm-AA}
Let $(X,T)$ be a minimal system. Then $(X,T)$ is a $d$-step almost
automorphic system for some $d\in \N\cup\{\infty\}$ if and only if
it is an almost one-to-one extension of its maximal $d$-step
nilfactor $(X_d, T)$.
\end{prop}

\subsection{Families}
Let $\mathcal{P}=\mathcal{P}(\mathbb{Z}_{+})$ be the collection of all subsets of $\mathbb{Z}_{+}$.
A subset $\mathcal{F}$ of $\mathcal{P}$ is a {\it family} if it is hereditary upwards,
i.e. $F_1\subset F_2$ and $F_1\in \mathcal{F}$ imply $F_2\in \mathcal{F}$.
A family $\mathcal{F}$ is {\it proper} if it is a
proper subset of $\mathcal{P}$, i.e. neither empty nor all of $\mathcal{P}$.
It is easy to see that $\mathcal{F}$ is proper if and only if
$\mathbb{Z}_{+}\in \mathcal{F}$ and $\emptyset \not\in\mathcal{F}$.
A family $\F$ has the \emph{Ramsey property}
if $F\in\F$ and $F=F_1\cup F_2$ imply that $F_i\in \F$ for some $i\in\{1,2\}$.
Any subset $\mathcal{A}$ of $\mathcal{P}$ generates a family
$$[\mathcal{A}]=\{F\in \mathcal{P}:F\supset A  \text{ for some } A\in \mathcal{A}\}.$$
If a proper family $\mathcal{F}$ is closed under finite intersection, then $\mathcal{F}$
is called a {\it filter}. For a family $\mathcal{F}$, the
{\it dual family} is
$$\mathcal{F}^{\ast}=\{F\in \mathcal{P}:\mathbb{Z}_{+}\setminus F\not\in \mathcal{F}\}
=\{F\in \mathcal{P}: F\cap F'\not=\emptyset \text{ for all }F'\in \mathcal{F}\}.$$
$\mathcal{F}^{\ast}$ is a family, proper if $\mathcal{F}$ is.
It is well known that a proper family has
the Ramsey property if and only if its dual $\F^*$ is a filter
\cite{F1981}.
Clearly, for a family $\F$
$$(\mathcal{F}^{\ast})^{\ast}=\mathcal{F} \text{ and }\mathcal{F}_1\subset
\mathcal{F}_2\Rightarrow \mathcal{F}_2^{\ast}\subset \mathcal{F}_1^{\ast}.$$

We say that a subset $F$ of $\mathbb{Z}_{+}$ is
\begin{enumerate}
\item \emph{thick} if it contains arbitrarily long blocks of consecutive integers, that is,
for every $d\geq 1$ there is $n\in\N$ such that $\{n,n+1,\dotsc,n+d\}\subset F$;
\item \emph{syndetic} if it has bounded gaps, that is, for some $N\in\N$ and every $k\in\N$ we have
$\{k,k+1,\dotsc,k+N\}\cap A\neq\emptyset$;
\item {\it piecewise syndetic} if it is the intersection of a syndetic set with
a thick set;
\item  \emph{thickly syndetic} if it has non-empty intersection with every piecewise syndetic set
\end{enumerate}

The collection of all syndetic (resp. thick) subsets is denoted by
$\mathcal{F}_s$ (resp. $\mathcal{F}_t$).
Note that $\mathcal{F}_s^*=\mathcal{F}_t$ and $\mathcal{F}_t^*=\mathcal{F}_s$.
The collection of all piecewise syndetic (resp. thickly syndetic) subsets
is denoted by $\mathcal{F}_{ps}$ (resp. $\mathcal{F}_{ts}$).

Let $\{b_i\}_{i\in I}$ be a finite or infinite sequence in $\mathbb{Z}_{+}$. One defines
$$FS(\{b_i\}_{i\in I}) =\{\underset{i\in \alpha}{\sum} b_i: \alpha
\text{ is a finite non-empty subset of }  I \}.$$
$F$ is an IP-set if it contains some $FS(\{p_i\}_{i=1}^{\infty})$
 where $p_i \in  \mathbb{N}$. The collection of all
IP-sets is denoted by $\mathcal{F}_{ip}$.
A subset of $\mathbb{Z}_{+}$ is called an $IP^{*}$-set, if it has non-empty
intersection with any IP-set. IP-sets are important in the study of dynamical properties,
see \cite{F1981,VB}.

If $I$ is finite, then one says $FS(\{ p_i \}_{i\in I})$ is an {\em
finite IP set} of length $|I|$.  The collection of all sets containing finite IP sets
with arbitrarily long lengths is denoted by $\F_{fip}$.

Let $E$ be a finite or infinite set in $\mathcal{P}(\Z_+)$, One defines
$$\Delta(E)=\{a-b: a\geq b, a,b \in E\}.$$
A subset $F$ of $\mathbb{Z}_{+}$ is called a {\it difference set}
if it contains some $\Delta(E)$ with $|E|$  infinite.
The collection of all difference sets is denoted by $\mathcal{F}_{\Delta}$.
A subset of $\mathbb{Z}_{+}$ is called a $\Delta^{*}$-set, if it has non-empty
intersection with any difference set. %The collection of all $\Delta^{*}$-sets is denoted by $\mathcal{F}_{\Delta}^{*}$.

If $E$ is a finite set, then one says that $\Delta(E)$ is a
{\em finite difference set} of length $|E|$.
The collection of all sets containing finite difference sets
with arbitrarily long lengths is denoted by $\F_{f\Delta}$.

\subsection{Technical lemmas}

Note that a factor map is {\it semi-open} if it sends any opene set to a set containing an opene set.
To end the section we state an easy lemma which follows from the continuity of $\pi$.
\begin{lem}\label{easy-check}
Let $\pi: (X,T)\lra (Y,S)$ be a semi-open factor map between two t.d.s. and $\F$ be a family.
If $(Y,S)$ is $\F$-sensitive (resp. block $\F$-sensitive, strongly $\F$-sensitive), so is $(X,T)$.
\end{lem}

%There are minimal distal t.d.s. which are strongly thick sensitive.

%Let $T: \T^2\lra \T^2$, $(x,y)\lra (x+\alpha,x+y)$. Then $T^n(x,y)=(x+n\alpha, nx+y+\frac{1}{2}n(n-1)\alpha).$
%Let $n_i$ go to infinity rapidly and $x_k=\sum_{i=k}^\infty \frac{i}{2^{n_i}}$. Then $x_k\ra 0$.
%A simple computation shows that
%$$d(T^n(0,0),T^n(x_k,0))\ge \frac{1}{4},\ n=2^{n_i-1}+j,\ 0\le j\le i,\ i\ge k.$$

%I do not know if the same holds for minimal distal systems which are not equicontinuous.
%We say $\pi:X\lra Y$ is quasi-distal if for each $y\in Y$
%and $x\in \pi^{-1}(y)$ there is $z\in \pi^{-1}(y)$
%such that $(x,z)$ is distal

%The question is that: for a minimal system what we can say?

The following lemma is easy to check.
\begin{lem}\label{natral-1} Let $(X,T)$ be a dynamical system, and $(\widetilde{X},\widetilde{T})$ be the natural extension of $(X,T)$.
Then $(X,T)$
is $\F$-sensitive (resp. block $\F$-sensitive, strongly $\F$-sensitive) if and only if $(\widetilde{X},\widetilde{T})$
is $\F$-sensitive (resp. block $\F$-sensitive, strongly $\F$-sensitive).
\end{lem}

The next lemma is from \cite[Proposition 4.4]{HKZ2014} or \cite[Lemma 2.4]{Eli-T}
%The following lemma was proved in \cite{5p}.

%\begin{lem} Let $(X,T)$ be minimal. Then $Z_2=X_{eq}$ implies that $Z_\infty=X_{eq}$.
%\end{lem}
\begin{lem}\label{ge0} Let $\pi:(X,T)\lra (Y,S)$ be a factor map with
$(X,T)$ minimal and $(Y,S)$ invertible. If $\pi$ is not almost
one-to-one, then $l=\inf_{y\in Y}\diam(\pi^{-1}(y))>0$.
\end{lem}

\section{Sensitivity for families}
To start our research we begin to study $\F$-sensitivity.
The goal is to show the notion of $\F$-sensitivity
is rough, meaning that for many families the notions are equivalent in the minimality setup.

Recall that the authors in  \cite{HKZ2014} proved that: a minimal system is either
$\F_t$-sensitive or an almost one-to-one extension of its maximal equicontinuous factor.
Moreover, they showed in \cite{HKKZ2014} that for minimal systems all of the following notions:
$\F_{ts}$-sensitivity, multi-sensitivity (see \cite{Moothathu2007} for a definition)
and $\F_t$-sensitivity
are equivalent. In this section,
we prove that for minimal systems all of the following notions:
$\F_{ts}$-sensitivity, $\F_{ip}$-sensitivity, $\F_{fip}$-sensitivity and $\F_{f\Delta}$-sensitivity are equivalent
(the equivalence to $\F_{Poin_d}$-sensitivity will be given in Section 5).

First we need a proposition which is basically due to Furstenberg \cite[Proposition 9.8]{F1981}.
Let $(X,T)$ be a t.d.s. and $\F$ be a family. Note that we say that $x\in X$ is $\F$-{\it recurrent},
if for each neighborhood
$U$ of $x$, $N(x,U)\in \F$; and $(X,T)$ is $\F$-recurrent if each point of $x\in X$ is $\F$-recurrent.

\begin{prop}\label{921-1} Let $(X,T)$ be a minimal equicontinuous system. Then $(X,T)$ is $\F_{f\Delta}^*$-recurrent.
\end{prop}
\begin{proof} Since $(X,T)$ is minimal and equicontinuous,
we can assume that $(X,T)$ is a Kronecker system.
That is, $X=G$, an abelian compact group, and  $Tx=ax$ for a fixed $a\in G$.
Let $x_0$ be any point of $X$ and $U$
be any open neighborhood of $x_0$.
Let $V$ be any neighborhood of $x_0$ such that $VV^{-1}x_0\subseteq U$.
Since $X$ is minimal, there are $l_1,\ldots, l_k\in\N$ such that
$\{a^{l_1}V,a^{l_2}V,\cdots,a^{l_k}V\}$ is a cover of $X$.

Let $\{S_n\}_{n=1}^m$ be  any finite sequence with $m>k$,
 then there are $a^{S_u},a^{S_v}$ contained in
the same subset $a^{l_t}V$. Then $a^{S_u-S_v}x_0\in U$,
which implies that $(X,T)$ is $\F_{f\Delta}^*$-recurrent.
\end{proof}

Using Proposition \ref{921-1} and some theorem in \cite{HKZ2014},
we have the following result.
\begin{thm}\label{sensitive} Let $(X,T)$ be minimal.
Then the following statements are equivalent:
\begin{enumerate}
\item $(X,T)$ is $\F_{ts}$-sensitive.
\item $(X,T)$ is $\F_{t}$-sensitive.
\item $(X,T)$ is $\F_{ip}$-sensitive.
\item $(X,T)$ is $\F_{fip}$-sensitive.
\item $(X,T)$ is $\F_{f\Delta}$-sensitive.
\item there exists $\delta>0$ such that for every $x\in X$
there is $y\in X$ such that $(x,y)$ is regional proximal and $d(x,y)>\delta$.
\item $(X,T)$ is not an almost one-to-one extension of $X_{eq}$.
\end{enumerate}
\end{thm}
\begin{proof}
It is clear that $\F_{ts}\subset \F_{t}\subset \F_{ip}\subset \F_{fip}\subset \F_{f\Delta}$.
By \cite[Theorem 3.1]{HKZ2014}, it remains to show (5) $\Rightarrow$ (7) and (7) $\Leftrightarrow$ (6)

(5) $\Rightarrow$ (7) Assume that $(X,T)$ is $\F_{f\Delta}$-sensitive with a sensitive constant $\delta>0$
and $\pi: (X,T)\lra (X_{eq},T_{eq})$ is almost
one-to-one.
Since
$(X_{eq},T_{eq})$ is a minimal equicontinuous system,
there is a compatible metric $d'$ such that $d'(T_{eq}x,T_{eq}y)=d'(x,y)$, for all
$x,y\in X_{eq}$.
Let  $y_0\in X_{eq}$ with $\pi^{-1}(y_0)$ singleton. We take an open set $W\subset X$
containing $\pi^{-1}(y_0)$ such that $\diam(W)<\delta$, and then there is an open set $V\subset X_{eq}$
containing $y_0$ such that $\pi^{-1}V\subset W$.

Let $B(y_0,\ep)\subset V$ for some $\ep>0$ and $U=\pi^{-1}(V_1)$ with $V_1=B(y_0,\ep/2)$.
By Proposition~\ref{921-1}, $N(y_0, V_1)\in \F_{f\Delta}^*$.

For $n\in N(y_0,V_1)$, we have $d'(T^n_{eq}y_0,y_0)<\ep/2$.
Since $d'(T^m_{eq}y,T^m_{eq}y_0)<\frac{\ep}{2}$ for all $m\in \mathbb{N}$ and $y\in V_1$,
we deduce that $T^n_{eq}(V_1)\subset V$ for $n\in N(y_0,V_1)$.
For $U=\pi^{-1}(V_1)$ and $n\in N(y_0,V_1)$ we get
$$T^n(U)=T^n\pi^{-1}(V_1)\subset \pi^{-1}(T^n_{eq}V_1)\subset \pi^{-1}(V)\subset W.$$
This means that $N(U,\delta)\cap N(y_0,V_1)=\emptyset$, which implies $N(U,\delta)\not\in \F_{f\Delta}$.

(6) $\Rightarrow$ (7) is obvious.

(7) $\Rightarrow$ (6) follows from Lemma~\ref{ge0}.
\end{proof}

\section{Block sensitivity and strong $\F_{fip},\ \F_{ip}$-sensitivity}
In this section we study block sensitivity and some related notions of strong sensitivity, and prove Theorems A, B and C.
This will be done in the following three subsections.

\subsection{Block \texorpdfstring{$\F_t$}{Ft}-sensitivity}
In this subsection, we discuss block $\F_t$-sensitivity and give a proof of Theorem A.

Recall that a t.d.s. $(X,T)$ is called \emph{block $\F_t$-sensitive} if there is $\delta>0$
such that for each $x\in X$, every neighborhood $U_x$ of $x$ and $l\in\N$ there are $y_l\in U_x$ with
$\{n\in\Z_+:d(T^nx,T^ny_l)>\delta\}$ containing $\{m+1,\ldots,m+l\}$ for some $m\in \N$. In fact we will
show the following theorem which covers Theorem~A.

\begin{thm}\label{thm:block-F-t-sensitive}
Let $(X,T)$ be a minimal dynamical system. Then the following conditions are equivalent:
\begin{enumerate}
  \item $(X,T)$ is block $\F_t$-sensitive;
  \item there exists $\delta>0$ such that for every $x\in X$
  there exists $y\in X$ such that $(x,y)$ is regional proximal and
  $\inf_{n\in\Z_+} d(T^nx,T^ny)>\delta$;
  \item $\pi\colon X\to X_{eq}$ is not proximal.
\end{enumerate}
\end{thm}

We start with

\begin{prop}\label{ye911-1} Let $(X,T)$ be a t.d.s. and $\pi:(X,T)\lra (X_{eq},T_{eq})$ be the factor map.  If $(X,T)$ is block $\F_t$-sensitive
then $\pi$ is not proximal. %quasi-distal.
\end{prop}
\begin{proof}
Let $d,d'$ be the compatible metrics of $X,X_{eq}$ respectively.
Let $\ep_k>0$ with $\ep_k\ra 0$.  Then for each $k\in\N$, there is $0<\tau_k,
\tau_k'<\ep_k$ such that
if $d'(w_1,w_2)<\tau_k$ with $w_1,w_2\in X_{eq}$ then $d'(T^i_{eq}w_1,T^i_{eq}w_2)<\ep_k$ for any $i\in\Z_+$; and
if $w_1,w_2\in X$ with $d(w_1,w_2)<\tau_k'$ then $d'(\pi(w_1),\pi(w_2))<\tau_k$.

Pick $x\in X$ and put $U_k=B(\tau_k',x)$.
%Since $x$ is minimal, $N(x,U_k)=\{a_1<a_2<\ldots\}$ is syndetic.
By the assumption $(X,T)$ is block $\F_t$-sensitive,
 thus for each $j\in \N$, there is $y_k^j\in U_k$ such that $F=\{n\in\Z_+:d(T^nx,T^ny_k^j)>\delta\}$
containing $\{a_k^j,a_k^j+1,\ldots, a_k^j+j\}$ (with $\delta$ the sensitive constant).
%We may assume that $F\supset \cup_{j=1}^\infty\{a_{i_j}+1,\ldots, a_{i_j}+j\}.$
%Note that $T^{a_{i_j}}x\in U_k$.$z_k^1\in B(\tau_k',x)$ and

Without loss of generality we assume that $T^{a_k^j}x\ra z_k^1$ and
$T^{a_k^j}y_k^j\ra z_k^2$ when $j\ra \infty$. It is clear that
$d(T^iz_k^1,T^iz_k^2)\ge \delta$
for each $i\in\Z_+$. Now let $z_1=\lim_{k\ra \infty}  z_k^1$ and $z_2=\lim_{k\ra \infty}  z_k^2$. We have $d(T^iz_1,T^iz_2)\ge \delta$
for each $i\in\Z_+$.

Now we show that $\pi(z_1)=\pi(z_2)$. Since $y_k^j\in U_k$, it is clear that $d'(\pi(x),\pi(y_k^j))<\tau_k$ and thus we have
$d'(T^i_{eq}\pi(x),T^i_{eq}\pi(y_k^j))<\ep_k$ for each $i\in\Z_+$. Particularly, $$d'(T^{a_k^j}_{eq}\pi(x),T^{a_k^j}_{eq}\pi(y_k^j))<\ep_k$$ for each $j\in\N$.
This implies that $d'(\pi(z_k^1),\pi(z_k^2))\le \ep_k$, and hence $d'(\pi(z_1),\pi(z_2))=0$.
We have proved that $\pi(z_1)=\pi(z_2)$. This indicates that $\pi$ is not proximal, finishing the proof.
\end{proof}

\begin{proof}[Proof of Theorem~\ref{thm:block-F-t-sensitive}]

 (1)$\Rightarrow$(3) follows from the above propostion.

(3)$\Rightarrow$(2)
There exists a regional proximal pair $(z_1,z_2)$ which is not proximal.
Let $\delta=\frac{1}{2}\inf_{n\in\Z_+} d(T^nz_1,T^nz_2)>0$.
Fix a point $x\in X$. As $z_1$ is a minimal point of $(X,T)$,
there exists a sequence of positive numbers $\{n_i\}$
such that $\lim_{i\to\infty}T^{n_i}z_1\to x$.
By the compactness of $X$, without loss of generality, assume that
$\lim_{i\to\infty}T^{n_i}z_2\to y$.
Then $(x,y)$ is regional proximal, since $Q(X,T)$ is closed and $T\times T$-invariant.
We also have $\inf_{n\in\Z_+} d(T^nx,T^ny)\geq \inf_{n\in\Z_+} d(T^nz_1,T^nz_2)>\delta$.

(2)$\Rightarrow$(1)
Fix $x\in X$ and a neighborhood $U$ of $x$ and $l\in \N$.
There exists $y\in X$ such that $(x,y)$ is regional proximal and
$\inf_{n\in\Z_+} d(T^nx,T^ny)>\delta$.
Choose small enough neighborhood $V\subset U$ of $x$ and neighborhood $W$ of $y$ such that
$\min_{0\le i\le l} d(T^iV,T^iW)>\frac{1}{2}\delta$

As $(x,y)$ is regional proximal, $N(x,W)$ is a $\Delta$-set \cite[Proposition 4.7]{HKZ2014}.
We also have that $N(V,V)$ is a $\Delta^*$-set \cite[Page 177]{F1981}.
Then $N(x,W)$ intersects $N(V,V)$.
Pick $n\in N(x,W)\cap N(V,V)$ and $x'\in V\cap T^{-n}V$.
Then $T^nx\in W$, $T^nx'\in V$.
This implies that $d(T^{n+i}x,T^{n+i}x')\geq\min_{0\le i\le l} d(T^iV,T^iW)>\frac{1}{2}\delta$ for $i=0,1,\ldots,l$.
Therefore, $(X,T)$ is block $\F_t$-sensitive.
\end{proof}

We have the following corollary.

\begin{cor} There is a minimal system which is $\F_t$-sensitive and not
strongly $\F_t$-sensitive.
\end{cor}
\begin{proof} There is a minimal system such that $\pi:X\lra X_{eq}$
is a proximal extension and not almost one-to-one extension \cite{GW1979}.
Then $(X,T)$ is $\F_t$-sensitive by \cite[Theorem 3.1]{HKZ2014},
and is not strongly $\F_t$-sensitive by Proposition \ref{ye911-1}.
\end{proof}

\subsection{Block \texorpdfstring{$\F_{ip}$}{Fip}-sensitivity
and strong \texorpdfstring{$\F_{fip}$}{Fip}-sensitivity}
In this subsection, we investigate block $\F_{ip}$-sensitivity, strong $\F_{fip}$-sensitivity and show Theorem  B.
In this subsection we assume that $T$ is a homeomorphism (since some results we use are stated for homeomorphisms and
it will take some pages to show they are true for continuous and surjective maps).

Recall that a t.d.s. $(X,T)$ is called \emph{ block $\F_{ip}$-sensitive} if there is $\delta>0$
such that for each $x\in X$, every neighborhood $U_x$ of $x$ and $l\in \N$ there is $y_l\in U$
such that $\{n\in\Z_+: d(T^nx,T^ny_l)>\delta\}$ contains a finite IP-set of length  $l$.
By the Ramsey property of $\F_{fip}$, an equivalent definition can be stated as follows:
there is $\delta>0$ such that for any opene $U$ of $X$ and $l\in \N$ there are $y_l,z_l\in U$
such that $\{n\in\Z_+: d(T^ny_l,T^nz_l)>\delta\}$ contains a finite IP-set of length $l$.
As before we will show the following theorem which covers Theorem~ B.

%choose $z_1\in U_1$, then there are
%$n_1< m_2< n_2 \in \mathbb{N}$ such that
%$$\mu(T^{-(p_{n_1+1}+p_{n_1+2}+\cdots +p_{m_2})}
%U_1\cap T^{-(p_{n_1+1}+p_{n_1+2}+\cdots +p_{n_2})}U_1)>0.$$
%Let $q_2=p_{m_2+1}+p_{m_2+2}+\cdots +p_{n_2}$, then $\mu(U_1\cap T^{-q_2}U_1)>0$
%which implies that $$U\cap T^{-q_1}U\cap T^{-q_2}U\cap T^{-q_1-q_2}U\not=\emptyset.$$
%Set $U_2=U_1\cap T^{-q_2}U_1$ and choose $z_1\in U_1$.
%Continue the process we find $q_i$ and $z_i\in U_i$ with
%$T^l(z_i)\in U$ for $l\in FS(\{q_j\}_{j=1}^i)$.

\begin{thm}\label{925--1}
Let $(X,T)$ be a minimal system.
Then the following statements are equivalent.
\begin{enumerate}
\item  $(X,T)$ is  strongly $\F_{fip}$-sensitive;
\item  $(X,T)$ is  block $\F_{ip}$-sensitive;
\item there exists $\delta>0$ such that for every $x\in X$
there exists $y\in X$ such that $(x,y) \in \RP^{[\infty]}$ with $d(x,y)\geq\delta$;
\item   $\phi: X\lra X_{\infty}$ is not almost one-to-one.
\end{enumerate}
\end{thm}

To prove Theorem~\ref{925--1} we need some preparation.
The following lemma is from \cite{Gillis}.
\begin{lem}\label{giliss}
Let $(X,\mathcal{B},\mu)$ be a
probability space, and $\{E_i\}_{i=1}^\infty$ be a sequence of
measurable sets with $\mu(E_i)\ge a>0$ for some constant $a$.
Then for any $k\ge 1$ and $\ep>0$ there is $N=N(a, k,
\ep)$ such that for any tuple $\{ s_1<s_2<\cdots <s_n\}$ with $n\ge
N$ there exist $1\le t_1<t_2<\cdots<t_k\le n$ with
\begin{align}\label{bds-key}
\mu(E_{s_{t_1}}\cap E_{s_{t_2}}\cap \cdots \cap E_{s_{t_{k}}})\ge
a^k-\ep.
\end{align}
\end{lem}

We will use the next lemma derived from Lemma \ref{giliss}.
\begin{lem}\label{s-times} Let $(X,T)$ be a t.d.s. with $\mu\in M(X, T)$. Let $U\in \mathcal{B}_X$ with $a=\mu(U)>0$.
Then there is $n=n(a)$ such that for any finite IP-set $FS(\{p_i\}_{i=1}^n)$ there is $q\in FS(\{p_i\}_{i=1}^n)$ such that $\mu(U\cap T^{-q}U)\ge \frac{1}{2}a^2$.
\end{lem}
\begin{proof} Apply Lemma \ref{giliss} to $k=2$, $\ep=\frac{1}{2}a^2$ and consider the finite tuple
$$T^{-p_1}U, \ldots, T^{-p_1-\ldots-p_n}U.$$
\end{proof}

The notion of {\it central set} was introduced in \cite{F1981}. It is known that a central set contains an IP-set \cite[Proposition 8.10]{F1981}.

\begin{prop}\label{925-1} Let $(X,T)$ and $(Y,S)$ be minimal.  If $\pi:X\lra Y$ is proximal
and not almost one-to-one, then $(X,T)$ is strongly $\F_{ip}$-sensitive.
\end{prop}
\begin{proof} By Lemma \ref{ge0}, $l=\inf_{y\in Y}\diam(\pi^{-1}(y))>0$. For each $y\in Y$, choose $x_1(y),x_2(y)\in \pi^{-1}(y)$
with $d(x_1(y),x_2(y))=l(y)\geq l$.

For $x\in X$, let $y=\pi(x)$. Then we have $d(x,x_1(y))\ge \frac{l}{2}$ or $d(x,x_2(y))\ge \frac{l}{2}$.
Without loss of generality, we assume that $d(x,x_1(y))\ge \frac{l}{2}$. Then $(x,x_1(y))$ is proximal.

Let $\delta=\frac{l}{8}$ and $U'$, $V$ be open neighborhoods of $x, x_1(y)$ with $diam(U'),\ diam(V)<\frac{l}{8}$ respectively.
Then $d(U',V)>\delta$. Choose a smaller $U$ with the same properties and $U'\supset \overline{U}$.
We know that $N(x,V)$ is a central set and hence it contains an IP-set $FS(\{p_i\}_{i=1}^{\infty})$. We are
going to show that there is $z\in U'$ such that $d(T^lx, T^lz)>\delta$ for all $l$ in a sub IP-set of $FS(\{p_i\}_{i=1}^{\infty})$.

To do this let $\mu\in M(X,T)$, then $a=\mu(U)>0$. Applying Lemma \ref{s-times} to $U$ there are $n_1$ and $q_1\in FS(\{p_i\}_{i=1}^{n_1})$ such that $\mu(U\cap T^{-q_1}U)\ge \frac{1}{2}a^2$.
% Since $\mu(U)>0$, there is $m_1< n_1 \in \mathbb{N}$
%such that $$\mu(T^{-(p_1+p_2+\cdots +p_{m_1})}U\cap T^{-(p_1+p_2+\cdots +p_{n_1})}U)>0.$$
%Let $q_1=p_{m_1+1}+p_{m_1+2}\cdots +p_{n_1}$, then $\mu(U\cap T^{-q_1}U)>0$. Set

Let $U_1=U\cap T^{-q_1}U$ and apply Lemma \ref{s-times} to $U_1$ there are $n_2$ and $q_2\in
FS(\{p_i\}_{i=n_1+1}^{n_2})$ such that $\mu(U_1\cap T^{-q_2}U_1)\ge \frac{1}{8}a^4$.
Note that we have $U\cap T^{-q_1}U\cap T^{-q_2}U\cap T^{-q_1-q_2}U\not=\emptyset.$

Inductively for any $k\in \N$ we obtain $n_1,\ldots,n_k$, $U_1,\ldots, U_k$ and $q_1,\ldots, q_k$
such that $q_{j+1}\in FS(\{p_i\}_{i=n_j+1}^{n_{j+1}})$,
$U_{j+1}=U_j\cap T^{-q_{j+1}}(U_j)$ with $\mu(U_{j+1})\ge C_j>0$ for $j=0,\ldots, k-1$ (set  $U_0=U$ and $n_0=0$).
This implies that
$$\mu(U\bigcap \bigcap_{l\in FS(\{q_i\}_{i=1}^k)}T^{-l}U)>0.$$

Thus, for each $k\in\N$ there is $z_k\in U$ such that $T^l(z_k)\in U$ for all $l\in FS(\{q_i\}_{i=1}^k).$
Without loss of generality, assume that $z=\underset{{k\ra \infty}}{\lim} z_k$, then $T^l(z)\in cl(U)\subset U'$
for $l\in FS(\{q_j\}_{j=1}^{\infty})$. We know $d(T^lx, T^lz)>\delta$ for $l\in FS(\{q_j\}_{j=1}^{\infty})$.
Since $FS(\{q_j\}_{j=1}^{\infty})\subseteq FS(\{p_j\}_{j=1}^{\infty})$,
this implies that $(X,T)$ is strongly $\F_{ip}$-sensitive.

%choose $z_1\in U_1$, then there are
%$n_1< m_2< n_2 \in \mathbb{N}$ such that
%$$\mu(T^{-(p_{n_1+1}+p_{n_1+2}+\cdots +p_{m_2})}U_1\cap T^{-(p_{n_1+1}+p_{n_1+2}+\cdots +p_{n_2})}U_1)>0.$$
%Let $q_2=p_{m_2+1}+p_{m_2+2}+\cdots +p_{n_2}$, then $\mu(U_1\cap T^{-q_2}U_1)>0$
%which implies that $$U\cap T^{-q_1}U\cap T^{-q_2}U\cap T^{-q_1-q_2}U\not=\emptyset.$$
%Set $U_2=U_1\cap T^{-q_2}U_1$ and choose $z_1\in U_1$.
%Continue the process we find $q_i$ and $z_i\in U_i$ with
%$T^l(z_i)\in U$ for $l\in FS(\{q_j\}_{j=1}^i)$.

\end{proof}

By \cite[Theorem 3.2]{SY2012} we know that for any $d\in\N$ and any minimal t.d.s. $(X,T)$, $(x,y)\in \RP^{[d]}(X)$ if and only if
for any neighborhood $V$ of $y$, $N(x,V)$ contains a finite IP-set of length $d+1$.
As $\RP^{[\infty]}(X)=\underset{d=1}{\overset{\infty}{\cap}}\RP^{[d]}(X)$, so we have
\begin{lem}\label{SY33} Let $(X,T)$ be minimal and $(x,y)\in X\times X$. Then  $(x,y)\in \RP^{[\infty]}(X)$ if and only if
for any neighborhood $V$ of $y$, $N(x,V)\in \F_{fip}$.
\end{lem}

With the help of the above lemma and Lemma \ref{s-times} we are able to show
\begin{prop}\label{925-3} Let $(X,T)$ be minimal and $\pi:X\lra X_\infty$ is not proximal.
Then $(X,T)$ is strongly $\F_{fip}$-sensitive.
\end{prop}
\begin{proof} Since $\pi$ is not proximal, there are $(x_1,x_2)\in R_\pi$ which is a distal pair. It follows that
$(x_1,x_2)\in \RP^{[\infty]}=\overset{\infty}{\underset{ d=1}{\cap}}\RP^{[d]}$ and $d(T^nx_1,T^nx_2)\geq l$ for
any $n\in \mathbb{N}$. Let $U,V$ be closed neighborhoods of $x_1,x_2$ with $diam(U),diam(V)<\frac{l}{4}$ respectively.
Then $d(U,V)>\frac{l}{2}$ and we let $\delta=\frac{l}{2}$.
By Lemma \ref{SY33}, $N(x_1,V)\in \mathcal{F}_{fip}$. We are going to show that there is $z\in U$ such that
$d(T^lx_1, T^lz)>\delta$ for all $l\in F\in \F_{fip}$ with $F\subset N(x_1,V)$.
%in a sub IP-set of $FS(\{p_i\}_{i=1}^{\infty})$. % be the sensitive constant.
% Moreover, choose a smaller $U$ with the same properties and $U'\supset \overline{U}$.
%in a finite IP-set contained in $N(x,V)$ of length $k$.

For $k=1$. Using the same argument in the of Proposition~ \ref{925-1}, we get ${n_1^1}\in \N$ such that for any
finite IP set of length ${n_1^1}$ with $FS\{p_i^1\}_{i=1}^{n_1^1}\subset N(x_1,V)$,
 there is $q_1^1\in FS\{p_i^1\}_{i=1}^{n_1^1}$ such that $\mu(U\cap T^{-q_1^1}U)\ge \frac{1}{2}a^2$. Set $U^1=U\cap T^{-q_1^1}U.$

For $k=2$.
Using the same argument in the of Proposition~ \ref{925-1} (with respect to $U^1$),
we get ${n_2^2}\in \N$ such that for any
finite IP set of length ${n_2^2}$ with $FS\{p_i^2\}_{i=1}^{n_2^2}\subset N(x_1,V)$,
there are
$q_1^2, q_2^2, q_1^2+ q_2^2 \in FS\{p_i^2\}_{i=1}^{n_2^2}$
such that if we set $U^2=U^1\cap T^{-q_1^2}U^1\cap T^{-q_2^2}U^1\cap T^{-q_1^2-q_2^2}U^1$ then  $\mu(U^2)>0$. So we have
$$\mu(U\cap T^{-q_1^1}\cap  T^{-q_1^2}U\cap T^{-q_2^2}U\cap T^{-q_1^2-q_2^2}U)>0.$$

Inductively,  for any $k\in \N$ we obtain  $n_1^j,\ldots,n_j^j$, $U^1,\ldots, U^j$ and $q_1^j,\ldots, q_j^j$
for $1\le j\le k$ such that

\noindent $\bullet$ for $0\le m\le j-1$, $q^{j}_{m+1}\in FS(\{p_i^j\}_{i=n_m^j+1}^{n^{j}_{m+1}})\subset N(x_1,V)$ (set $n_0^j=0$).

\noindent $\bullet$
$U^{j+1}=U^j\cap \cap_{l\in FS\{q_i^j\}_{i=1}^j}T^{-l}U^j$ for $0\le j\le k-1$ satisfy that $\mu(U^{j+1})>0$ (set $U^0=U$). So we have
$\mu(A_k)>0$, where
$$A_k=U \bigcap \bigcap_{j=1}^k \bigcap_{l\in FS\{q_i^j\}_{i=1}^j}T^{-l}U.$$

Set $F=\cup _{k=1}^\infty  FS\{q_i^k\}_{i=1}^k$. Then $F\subset N(x_1,V)$ and $F\in \F_{fip}$. Take $z\in \cap_{k=1}^\infty A_k$,
then $T^lz\in U$ for all $l\in F$. This implies that $d(T^lx_1, T^lz)>\delta$ for all $l\in F\in \F_{fip}$.

For $u \in X$ there is a sequence $\{n_i\}$ such that $T^{n_i}x_1\rightarrow u$ and $T^{n_i}x_2\rightarrow v.$
Then $(u,v)\in \RP^{[\infty]}$ and $(u,v)$ is a distal pair with $d(T^nu,T^nv)\geq l$. Let $W, W'$ be closed
neighborhoods of $u$ and $v$ respectively with $diam(W), diam(W')<\frac{l}{4}$.
By the proof above, we know that there is $w\in W$ such that
$d(T^lu, T^lw)>\delta$ for all $l\in F$, where $F\in \F_{fip}$ with $F\subset N(u, W')\in \F_{fip}$.
So we have proved that $(X,T)$ is strongly $\F_{fip}$-sensitive.
\end{proof}

The following lemma is well known.
\begin{lem}\label{finiteip-2} Let $F$ be a finite IP-set of length $n$ and $F=F_1\cup F_2$. Then there is $i\in\N$ such that
$F_i$ is a finite IP-set of length $l(n)$ with $l(n)\lra \infty$ when $n\lra \infty$. This also implies that
$\F_{fip}$ has the Ramsey property.
\end{lem}

To end the proof we need another proposition.
\begin{prop}\label{925-4} Let $(X,T)$ be a minimal block $\F_{ip}$-sensitive t.d.s. with the sensitive constant $10\delta$.
Assume that $x\in X$ and $U$ is any neighborhood of $x$.
Then there are $z\in U$ and $y\in X$ such that $(y,z)\in \RP^{[\infty]}$ with $d(z,y)\geq\delta$.
\end{prop}
\begin{proof} Since $(X,T)$ is block $\F_{ip}$-sensitive, there is $\delta>0$ such that
for any $x_0\in X$, any neighborhood $V$ of $x_0$ and any $n\in \N$ there is $y_0,z_0\in V$
such that $\{m\in\N: d(T^my_0,T^mz_0)>10\delta\}$ contains a finite IP-set of length $n$.

\medskip

Let $U_0=B(x,4\delta)$ and $U_1=B(x,\delta)$. Without loss of generality,
we assume $U\subseteq U_1$. Then for  $n_1\in \N$  large enough there are $x_1^1,x_2^1\in U$
such that $F_1=\{n\in\N: d(T^nx_1^1,T^nx_2^1)>10\delta\}$ contains a finite IP-set of length $n_1$.
By the method of Proposition \ref{925-1}, there is $z_1\in U$ satisfying $T^nz_1\in U$ for
$n\in F_1'\subseteq F_1$, where $F_1'$ is a finite IP-set of length $k(n_1)$. Then
$d(T^nx_1^1,T^nz_1)>5\delta$ or $d(T^nx_2^1,T^nz_1)>5\delta$ for $n\in F_1'$.
Without loss of generality, we assume that $d(T^nx_1^1,T^nz_1)>5\delta$  for $n\in F_1''\subseteq F_1'$,
where $F_1'$ is a finite IP-set of length $l(k(n_1))$ (Lemma \ref{finiteip-2}). Then $T^nx_1^1\not \in U_0$ for $n\in F_1''$.
Let $U_2\subset U$ an open neighborhood of $x_1^1$ with diameter small enough such that $T^n\overline{U_2}\cap U_0=\emptyset$ for $n\in F_1''$.

\medskip

Then for  $n_1 \ll n_2\in \N$ large enough there are $x_1^2,x_2^2\in U_2$
such that $F_2=\{n\in\N: d(T^nx_1^2,T^nx_2^2)>10\delta\}$ contains a finite IP-set of length $n_2$.
By the method of Proposition~\ref{925-1} again,
there is $z_2\in U_2$ satisfying $T^nz_2\in U_2$ for $n\in F_2'\subseteq F_2$, where $F_2'$
is a finite IP-set of length $k(n_2)$. Then $d(T^nx_1^2,T^nz_2)>5\delta$ or
$d(T^nx_2^2,T^nz_2)>5\delta$ for $n\in F_2'$. Without loss of generality, we assume
$d(T^nx_1^2,T^nz_2)>5\delta$  for $n\in F_2''\subseteq F_2'$, where $F_2'$ is a
finite IP-set of length $l(k(n_2))$. Then $T^nx_1^2\not \in U_0$.
Let $U_3\subset U_2$ an open neighborhood of $x_1^2$ with diameter small enough such
that $T^n\overline{U_3}\cap U_0=\emptyset$ for $n\in F_2''$.

\medskip

Continue the process, we get $F_k\supseteq F_k'\supseteq F_k'',\ z_k,\ n_k$ and $U_k$ with $diam(U_k)\ra 0$ as $k\ra \infty$.
We have

(1) $d(T^nz_k,T^nx_1^k)\geq 5\delta$ for $n\in F_k''$ with $z_k\in U_k,\ x_1^k\in U_{k+1}\subseteq U_k$;

(2) $T^n\overline{U_{k+1}}\cap U_0\neq \emptyset$ for $n\in F_k''$.

Assume that $\underset{k\rightarrow \infty}{\lim}z_k= z$,
then $\underset{k\rightarrow \infty}{\lim}x_1^k= z$. Since $z\in \underset{k=1}{\overset{\infty}{\cap}}\overline{U_k}$,
we have $T^nz \not \in U_0$ for $n\in F_k''$. Thus, $N(z,U_0^c)\in \mathcal{F}_{fip}$, for $l(k(n))\ra \infty$ as
$n\ra \infty$.

Let $W=B(x,3\delta)$. Since $U_0^c$ is compact, we can cover $U_0^c$
by finitely many closed balls $\{V_1^1,V_2^1,\cdots,V_{l_1}^1\}$ with diameter less than $1$ and
$\underset{k=1}{\overset{l_1}{\cup}}V_{k}^1\subset W^c$.  By the Ramsey property of $\mathcal{F}_{fip}$,
there is $1\le m_1\le l_1$ such that $N(z,V_{m_1}^1)\in \mathcal{F}_{fip}$. Since $V_{m_1}^1$ is compact,
we can cover $V_{m_1}^1$ by finitely many closed balls $\{V_1^2,V_2^2,\cdots,V_{l_2}^2\}$ with diameter less than $\frac{1}{2}$ and
$\underset{k=1}{\overset{l_2}{\cup}}V_{k}^2\subset W^c$.  By the Ramsey property of $\mathcal{F}_{fip}$ again,
there is $1\le m_2\le l_2$ such that $N(z,V_{m_2}^2)\in \mathcal{F}_{fip}$. Continue the process, we get $V_{m_k}^k$
such that
$$N(z,V_{m_k}^k)\in \mathcal{F}_{fip},\ \diam(V_{m_k}^k)\leq\frac{1}{k}\ \text{and}\ V_{m_k}^k\subset W^c.$$
Let $y\in \underset{k=1}{\overset{\infty}{\cap}}V_{m_k}^k$. Then for any open neighborhood $W'$ of $y$, we have $N(z,W')\in \mathcal{F}_{fip}$
since $W'$ contains $V_{m_k}^k$ for some $k\in\N$. Lemma~\ref{SY33} implies that $(y,z)\in \RP^{[\infty]}$.
Since $y\in W^c$ and $z\in \overline{U_1}$, we conclude that $d(z,y)\geq 2\delta>\delta$.
This ends the proof.
\end{proof}

\begin{proof}[Proof of Theorem~\ref{925--1}]
(1) $\Rightarrow $(2) is obvious.

(2)$\Rightarrow$ (3)
Assume that $(X,T)$ is block $\F_{ip}$-sensitive.
Fix $x\in X$.
By Proposition~\ref{925-4}
for every $n\in\N$, there exists $x_n\in B(x,\frac{1}{n})$ and $y_n\in X$
such that $d(x_n,y_n)\geq\delta$ and $(x_n,y_n)\in \RP^{[\infty]}(X)$.
Without loss of generality, assume that $y_n\to y$.
Then $d(x,y)\geq \delta$ and $(x,y)\in \RP^{[\infty]}(X)$ as $\RP^{[\infty]}(X)$ is closed.

(3) $\Rightarrow$ (4) is obvious.

(4) $\Rightarrow$ (1) Since $\phi$ is not almost one-to-one, $\phi$ is either not proximal, or
proximal and not almost one-to-one. If $\phi:X\lra X_\infty$ not proximal, then by
Proposition~ \ref{925-3} we get that $(X,T)$ is strongly $\F_{fip}$-sensitive.
If $\phi:X\lra X_{\infty}$ is proximal, not almost one-to-one, by Proposition \ref{925-1} we get $(X,T)$ is strongly $\F_{fip}$-sensitive.
\end{proof}

\subsection{Strong \texorpdfstring{$\F_{ip}$}{Fip}-sensitive}
In this subsection, we study strong $\F_{ip}$-sensitivity and give the proof of Theorem C.
Recall that  we say a t.d.s. $(X,T)$ is strongly $\F_{ip}$-sensitive if there is $\delta>0$
such that for each opene subset $U$ of $X$, there are $x,y\in U$ with
$\{n\in\Z_+:d(T^nx,T^ny)>\delta\}\in \F_{ip}$. In fact we will show a stronger form of Theorem~C.

\begin{thm}\label{thm:Fip-sensitive}
Let $(X,T)$ be a minimal system. Then the following conditions are equivalent:
\begin{enumerate}
  \item $(X,T)$ is strongly $\F_{ip}$-sensitive;
  \item there is $\delta>0$ such that for every non-empty open subset $U$ of $X$
  there exists a proximal pair $(x,y)$ with $x\in U$ and $d(x,y)>\delta$;
  \item $\pi\colon X\to X_D$ is not almost one-to-one,
  where $(X_D,T)$ is the maximal distal factor of $(X,T)$.
\end{enumerate}
\end{thm}

We say that $x$ is \emph{strongly proximal} to $y$ if $(y,y)\in \omega((x,y),T\times T)$, where $\omega(x,y)$
is the $\omega$-limit set of $(x,y)$.
Note that if $(x,y)$ is proximal and $y$ is a minimal point, then $x$ is strongly proximal to $y$.
We need two results from \cite{L12}.

\begin{lem}[\mbox{\cite[Lemma 4.8]{L12}}] \label{lem:strongly-proximal}
Let $(X,T)$ be a dynamical system and $x,y\in X$.
Then $x$ is strongly proximal to $y$ if and only if
for every neighborhood $U$ of $y$, $N(x,U)\cap N(y,U)$ contains an IP-set.
\end{lem}

\begin{prop}[\mbox{\cite[Proposition 5.9]{L12}}] \label{prop:IP-set-stronlgy-proximal}
Let $(X,T)$ be a dynamical system, $x\in X$ and $Y\subset X$ be a closed subset of $X$.
If $N(x,Y)$ contains an IP set, then there exists $y\in Y$
such that $x$ is strongly proximal to $y$.
\end{prop}

Now we show a proposition.

\begin{prop}\label{929-1}
Let $(X,T)$ be a minimal system.
Then $(X,T)$ is strongly $\F_{ip}$-sensitive
if and only if there is $\delta>0$ such that every non-empty open subset $U$ of $X$,
there is $x\in U$ and $y\in X$ with $d(x,y)>\delta$ and $x$ is strongly proximal to $y$.
\end{prop}
\begin{proof}
First assume that $(X,T)$ is strongly $\F_{ip}$-sensitive with sensitive constant $8\delta>0$.
Fix a non-empty open subset $U$ of $X$. Pick $z\in U$ and let $V=U\cap B(z,\delta)$.
There are $x_1,x_2\in V$ such that $F=\{n\in\Z_+\colon d(T^nx_1,T^nx_2)>8\delta\}$ contains an IP-set.
Let $W=X\setminus B(z,2\delta)$.
By the Ramsey property of $\F_{ip}$, either $N(x_1,W)$ or $N(x_2,W)$ contains an IP-set.
By Proposition~\ref{prop:IP-set-stronlgy-proximal}
there exists $y\in W$ such that either $x_1$ or $x_2$ is strongly proximal to $y$.
It is clear that $d(x_1,y)>\delta$ and $d(x_2,y)>\delta$.

Now we show the sufficiency.  Fix a non-empty open subset $U$ of $X$.
there is $x\in U$ and $y\in X$ with $d(x,y)>\delta$ and $x$ is strongly proximal to $y$.
By Lemma~\ref{lem:strongly-proximal},
$N(x, B(y,\delta/3))$ contains an IP-set $FS(\{p_i\}_{i=1}^{\infty})$.
By the method of Proposition~\ref{925-1},
there exist $z\in B(x,\delta/3)$ and an IP subset $FS(\{q_j\}_{j=1}^{\infty})$
such that $FS(\{q_j\}_{j=1}^{\infty})\subset N(z,B(x,\delta/3))$ and $FS(\{q_j\}_{j=1}^{\infty})\subset FS(\{p_i\}_{i=1}^{\infty})$.
Then $FS(\{q_j\}_{j=1}^{\infty})\subset \{n\in\Z_+\colon d(T^nx,T^nz)>\delta/3\}$,
which implies that $(X,T)$ is strongly $\F_{ip}$-sensitive
with the sensitive constant $\delta/3$.
\end{proof}

We are in the position to give:

\begin{proof}[Proof of Theorem~\ref{thm:Fip-sensitive}]
(1)$\Rightarrow$(2) follows from the Proposition~\ref{929-1}.

(2)$\Rightarrow$(3) For every point $x\in X$, there exists a sequence $y_n$ and $z_n$
such that $\lim_{n\to \infty}y_n=x$ and $(y_n,z_n)$ is proximal and $d(y_n,z_n)>\delta$.
Without loss of generality, assume that $\lim_{n\to\infty}z_n=z$.
Then $d(x,z)\geq \delta$. Note that $(x,z)\in S_{D}$,
where $S_{D}$ is the distal relation, $X/S_D=X_D$.
Let $\pi\colon X\to X_D$. Then $\{x,z\}\in \pi^{-1}(\pi(x))$.
So $\pi$ is not almost one-to-one.

(3)$\Rightarrow$(1)
If $\pi$ is proximal, then by Proposition \ref{925-1},
$(X,T)$ is strongly $\F_{ip}$-sensitive.
So we assume that $\pi$ is not proximal. This implies that $P(X)$ is not closed.
So there is a distal pair $(y,z)$ and
proximal pairs $(y_i,z_i)$ such that $(y_i,z_i)\lra (y,z)$.
Let $\inf_{n\in\Z_+}d(T^ny,T^nz)=4\delta$.

Fix a non-empty open subset $U$ of $X$.
As $y$ is a minimal point, there exists $k\in\N$ such that $T^ky\in U$.
There exists $n\in \N$ such that $T^ky_n\in U\cap B(T^ky,\delta)$ and $d(T^kz_n,T^kz)<\delta$.
Let $x_1=T^ky_n$ and $x_2=T^kz_n$.
Then $x_1\in U$, $d(x_1,x_2)>\delta$ and $(x_1,x_2)$ is proximal.
As $x_2$ is a minimal point, $x_1$ is strongly proximal to $x_2$.
Then the result follows from Proposition~\ref{929-1}.
\end{proof}

\section{Strong sensitivity for other families}\label{section5}

In this section we study strong sensitivity for other families and shall prove Theorem~ D. Namely, we
will investigate the properties of strong $\F_t$- and strong $\F_{Poin_d}$-sensitivity.

%In fact we do more, namely

%we also investigate the properties of strong $\F_{Poin_d}$-sensitivity.

\subsection{Strong \texorpdfstring{$\F_t$}{Ft}-sensitivity}

In this subsection, we discuss strong $\F_t$-sensitivity, and prove Theorem D.
%give some necessary or sufficient conditions.
%for determination of a dynamic system of strong thick sensitivity.
Recall that for a t.d.s. $(X,T)$, we say $(X,T)$ is strongly $\F_t$-sensitive if there is $\delta>0$
such that for each opene subset $U$ of $X$, there are $x,y\in U$ with
$\{n\in\Z_+:d(T^nx,T^ny)>\delta\}\in\F_t$.
So $(X,T)$ is not strongly $\F_t$-sensitive if there are $\delta_n\lra 0$ and opene subsets $U_n$
such that for any $x_n,y_n\in U_n$, there is a syndetic subset $F$ of $\Z_+$ with
$d(T^mx_n,T^my_n)\le \delta_n$ for all $m\in F$.

%We will first study basic properties of strong $\F_t$-sensitivity and then give some examples.

%\subsubsection{Basic properties of strong $\F_t$-sensitivity}

%\begin{thm}\label{961-1} Let $(X,T)$ be a minimal system.
%Then the following conditions are equivalent:
%\begin{enumerate}
%  \item $(X,T)$ is not strongly $\F_{t}$-sensitive;
%  \item $\pi\colon X\to X_D$ is proximal,
%  where $(X_D,T)$ is the maximal distal factor of $(X,T)$.
%\end{enumerate}
%\end{thm}

To prove Theorem D, we first show that strong $\F_t$-sensitivity
passes through proximal extensions.

\begin{prop}\label{1009-2} Let $\pi:(X,T)\lra (Y,S)$ be a proximal extension of minimal systems. If $(Y,S)$ is not strongly $\F_t$-sensitive,
then  neither is $(X,T)$.
\end{prop}
\begin{proof}Let $d,d'$ be the compatible metrics of $X,Y$ respectively.
Since $(Y,S)$ is not strongly $\F_t$-sensitive,
there are $\delta_k \ra 0$ and opene subsets
$U_k$ of $Y$ such that if $x_k,y_k\in U_k$ then there is a syndetic subset $F$ (depends on $x_k,y_k$) with $d'(S^nx_k,S^ny_k)<\delta_k$ for every $n\in F$.

Assume the contrary that $(X,T)$ is strongly $\F_t$-sensitive with a sensitive constant $\delta>0$.
Then for each opene subset $U$, there are $x,y\in U$ such that $\{n\in\Z_+: d(T^nx,T^ny)>\delta\}\in \F_t$.
Thus, there are $u_k,v_k\in \pi^{-1}(U_k)$ such that $F_k:=\{n\in \mathbb{Z}_+:  d(T^nu_k,T^nv_k)> \delta\}\in \F_t$.
Note that $E_k:=\{n\in\Z_+:d'(S^n\pi(u_k),S^n\pi(v_k))<\delta_k\}$ is a syndetic set.
This implies that there exists $b_k\in\N$ such that $$d'(S^{b_k}\pi(u_k),S^{b_k}\pi(v_k))<\delta_k \
 \text{and} \ d(T^{j}(u_k),T^{j}(v_k))>\delta$$ for $j\in [b_k-k,b_k+k]$.
Without loss of generality, assume that $T^{b_k}u_k\ra u,T^{b_k}v_k\ra v$. Then $d(T^nu,T^nv)\geq \delta, \forall n\in \Z_+$. Since
$\pi(T^{b_k}u_k)\ra \pi(u)$, $\pi(T^{b_k}v_k)\ra \pi(v)$ and $d'(S^{b_k}\pi(u_k),S^{b_k}\pi(v_k))<\delta_k$, we conclude that
$\pi(u)=\pi(v)$, a contradiction. This indicates that $(X,T)$ is not strongly $\F_t$-sensitive, ending the proof.
\end{proof}

\begin{prop}\label{961-2} Let $(X,T)$ be a minimal system. If $\pi\colon X\to X_D$ is  proximal,
then $(X,T)$ is not strongly $\F_{t}$-sensitive.
\end{prop}
\begin{proof}
By Proposition~ \ref{1009-2} and Theorem \ref{925-1}.
\end{proof}

To prove the converse of Theorem D, we need the structure theorem.
So we assume that $T$ is a homeomorphism first. When $(X,T)$ is not invertible,
we use natural extension to prove Theorem D.

Recall that an extension $\pi : X \to Y$ of minimal systems is a {\em
relatively incontractible (RIC) extension}\ if it is open and for
every $n \ge 1$ the minimal points are dense in the relation
$$
R^n_\pi = \{(x_1,\dots,x_n) \in X^n : \pi(x_i)=\pi(x_j),\ \forall \
1\le i \le j \le n\}.
$$

We say that a minimal system $(X,T)$ is a {\em strictly PI system}
if there is an ordinal $\eta$ (which is countable when $X$ is
metrizable) and a family of systems
$\{(W_\iota,w_\iota)\}_{\iota\le\eta}$ such that (i) $W_0$ is the
trivial system, (ii) for every $\iota<\eta$ there exists a
homomorphism $\phi_\iota:W_{\iota+1}\to W_\iota$ which is either
proximal or equicontinuous, (iii) for a limit ordinal $\nu\le\eta$
the system $W_\nu$ is the inverse limit of the systems
$\{W_\iota\}_{\iota<\nu}$,  and (iv) $W_\eta=X$. We say that $(X,T)$
is a {\em PI-system} if there exists a strictly PI system $\tilde X$
and a proximal homomorphism $\theta:\tilde X\to X$.

We have the following structure theorem for minimal systems

\begin{lem}[Structure theorem for minimal systems, \cite{EGS}]\label{structure}
Given a homomorphism $\pi: X \to Y$ of minimal dynamical system,
there exists an ordinal $\eta$ (countable when $X$ is metrizable)
and a canonically defined commutative diagram (the canonical
PI-Tower)
\begin{equation*}
\xymatrix
         {X \ar[d]_{\pi}             &
      X_0 \ar[l]_{{\theta}^*_0}
          \ar[d]_{\pi_0}
          \ar[dr]^{\sigma_1}         & &
      X_1 \ar[ll]_{{\theta}^*_1}
          \ar[d]_{\pi_1}
          \ar@{}[r]|{\cdots}         &
      X_{\nu}
          \ar[d]_{\pi_{\nu}}
          \ar[dr]^{\sigma_{\nu+1}}       & &
      X_{\nu+1}
          \ar[d]_{\pi_{\nu+1}}
          \ar[ll]_{{\theta}^*_{\nu+1}}
          \ar@{}[r]|{\cdots}         &
      X_{\eta}=X_{\infty}
          \ar[d]_{\pi_{\infty}}          \\
         Y                 &
      Y_0 \ar[l]^{\theta_0}          &
      Z_1 \ar[l]^{\rho_1}            &
      Y_1 \ar[l]^{\theta_1}
          \ar@{}[r]|{\cdots}         &
      Y_{\nu}                &
      Z_{\nu+1}
          \ar[l]^{\rho_{\nu+1}}          &
      Y_{\nu+1}
          \ar[l]^{\theta_{\nu+1}}
          \ar@{}[r]|{\cdots}         &
      Y_{\eta}=Y_{\infty}
     }
\end{equation*}
where for each $\nu\le\eta, \pi_{\nu}$ is RIC, $\rho_{\nu}$ is
isometric, $\theta_{\nu}, {\theta}^*_{\nu}$ are proximal and
$\pi_{\infty}$ is RIC and weakly mixing of all orders. For a limit
ordinal $\nu ,\  X_{\nu}, Y_{\nu}, \pi_{\nu}$ etc. are the inverse
limits (or joins) of $ X_{\iota}, Y_{\iota}, \pi_{\iota}$ etc. for
$\iota < \nu$.

Thus if $Y$ is trivial, then $X_\infty$ is a proximal extension of
$X$ and a RIC weakly mixing extension of the strictly PI-system
$Y_\infty$. The homomorphism $\pi_\infty$ is an isomorphism (so that
$X_\infty=Y_\infty$) if and only if  $X$ is a PI-system.
\end{lem}

\begin{lem}\cite[Lemma 7.16]{DSY12}\label{dsy2012}
Let $\pi: X\lra Y$  be a weakly mixing and RIC extension of minimal
systems. Then there is a dense $G_\delta$ subset $Y_0$ of $Y$ such that, for each $y\in Y_0$
and each $x\in\pi^{-1}(y)$, $P_{\pi}[x]$ is dense in $\pi^{-1}(y)$,
where $P_{\pi}[x]=\{z\in \pi^{-1}(\pi(x)): (x,z)\in P(X)\}.$
\end{lem}

\begin{thm}\label{1009-4} Let $(X,T)$ be minimal. If $(X,T)$ is not strongly $\F_t$-sensitive,
then $(X,T)$ is PI.
\end{thm}
\begin{proof} First we claim: if $(X,T)$ is minimal, and there is $x\in X$ such that $(x,y)$ is a
distal pair, and $y$ is proximal to $z_i\in X$ with $z_i\ra x,\ z_i\not=x,\ i\in\N$, then $(X,T)$ is  strongly $\F_t$-sensitive.

Let $\delta=\frac{1}{3}\inf_{n\in \N} d(T^nx,T^ny)$ and fix an opene set $U$ of $X$. Then there is
$l\in\N$ with $T^lx\in U$ by the minimality of $X$. This implies that $(T^lx,T^ly)$ is a distal pair
and $T^ly$ is proximal to $T^lz_i$ with
$T^lz_i\ra T^lx$, $T^lz_i \not =T^lx$. There is $i\in \N$ such that $T^lz_i\in U$.
Since $(T^ly,T^lz_i)$ is proximal,
we get that $\{n\in \Z_+: d(T^{n+l}y,T^{n+l}z_i)<\delta\}\in \F_t$.
This implies that $\{n\in \N: d(T^{n+l}x,T^{n+l}z_i)>\delta\}\in \F_t$.
We conclude that $(X,T)$ is strongly $\F_t$-sensitive, finishing the proof of the claim.

\medskip

Assume that $(X,T)$ is not PI. By Lemma \ref{structure},
  $\theta^*:X_\infty \ra X$ is proximal, $\pi_\infty:X_\infty\ra Y_\infty$ is weakly mixing, RIC and not an isomorphism.
By Lemma \ref{dsy2012}, there are $s\in Y_\infty$  and $u\in \pi_\infty^{-1}(s)$ such that $P_{\pi_\infty}[u]$ is
dense in the  $\pi_\infty^{-1}(s)$.  Since $\pi_\infty$ is not proximal, there is $v\in X_\infty$ such that $(u,v)$
is distal and $\pi_\infty(v)=\pi_\infty(u)$. Since $\theta^*$ is proximal, we know that
$(\theta^*(v),\theta^*(u))$ is distal. As $P_{\pi_\infty}[u]$ is dense in the  $\pi_\infty^{-1}(s)$,  there
are $v_i\ra v$ such that $v_i\not= v$ and $(v_i,u)$ is proximal. This implies that $(\theta^*(v_i),\theta^*(u))$ is proximal.
It is clear that $\theta^*(v_i)\not=\theta^*(v)$ and $\theta^*(v_i)\ra \theta^*(v)$.
Applying the claim we just proved, we conclude that $(X,T)$ is strongly $\F_t$-sensitive, a contradiction.
\end{proof}

Before proving the following key result for Theorem D we need two well known lemmas.

\begin{lem} \label{limit} Let $\pi:X\lra Y$ be an open factor map between two t.d.s. Assume that $y\in Y$
and $y_i\ra y$. Then for any $z\in \pi^{-1}(y)$ there are $z_i\in \pi^{-1}(y_i)$ such that $\lim z_i=z$.
\end{lem}

Let $E(X,T)$ be the enveloping semigroup of $(X,T)$.

\begin{lem}\label{surjection} Let $\pi:X\lra Y$ be a distal factor map between two minimal t.d.s.
Then $\pi$ is open and $\pi^{-1}(py)=p\pi^{-1}(y)$ for any $y\in Y$ and any $p\in E(X)$.
\end{lem}

\begin{thm}\label{2-1-1} \label{very-i}Let $(X_3,T)$ be minimal and $X_1\overset{\pi_1}{\leftarrow}X_2
\overset{\pi_2}{\leftarrow}X_3$, where $\pi_1$ is a non-trivial proximal
extension, $\pi_2$ is a non-trivial distal extension and $X_1$ is distal.
If $P(X_3)$ is not closed then $X_3$ is strongly $\F_t$-sensitive.
\end{thm}
\begin{proof}
Since $P(X_3)$ is not closed, there are a distal pair $(x_1,x_2)\in X_3\times X_3$  and proximal pairs $(x_1(i),x_2(i))\in X_3\times X_3$
for all $i\in\N$ such that $(x_1(i),x_2(i))\lra (x_1,x_2)$. Let $\pi=\pi_1\pi_2$.
It is clear that $\pi(x_1(i))=\pi(x_2(i))$
since $X_1$ is distal. This implies that $\pi(x_1)=\pi(x_2)$. Moreover, we may assume that $(x_1,x_2)$ is a minimal
point. As $(\pi_2(x_1),\pi_2(x_2))$
is proximal and minimal we know that $\pi_2(x_1)=\pi_2(x_2).$ Put $\delta=\inf_{n\in\Z_+}d(T^nx_1,T^nx_2)$ and
$U_i$ be an open neighborhood of $x_i$ with $\diam(U_i)<\delta/6$, $1\le i\le 2$.

Set $y=\pi_2(x_1)$ and $y_i=\pi_2(x_1(i)),\ i\in\N$. Then $\lim_{i\ra\infty}y_i=y$.
Let $$M=\overline{orb((x_1,x_2), T\times T)}\ \text{and}\ K=\{x\in X_3:(x_1,x)\in M\}\subset \pi_2^{-1}(y).$$
It is clear that $x_2\in K$. Moreover, $M$ is a minimal subsystem of $X_3\times X_3$
and for any $(z_1,z_2)\in M$ we have that $\pi_2(z_1)=\pi_2(z_2)$.
%Put $p:\Delta_{X_2}\lra X_2, (x,x)\mapsto x$. Note that $q=p\circ \pi_2\times
%\pi_2:M\lra X_2$ is distal and $q^{-1}(y)=\{x_1\}\times K$.
Let $p:M\lra X_3$ be the projection to the first coordinate. Then $p^{-1}(x_1)=\{x_1\}\times K$ and $p$ is a distal extension.
Put $p^{-1}(x_1(i))=\{x_1(i)\}\times K_i,\ i\in\N$.

Since $M\cap (U_1\times U_2)$ is an open neighborhood of $(x_1,x_2)$ and $p$ is open, by Lemma \ref{limit} there are
$x'_2(i)\in K_i$ such that $(x_1(i), x'_2(i))\in M\cap (U_1\times U_2)$ since $\lim_{i\ra\infty}x_1(i)=x_1$. Note that $x'_2(i)\in K_i$
and thus $\pi_2(x'_2(i))=y_i$.

We can choose a sequence $\{n_k\}$ such that $T^{n_k}(x_1)\lra (x_1(i))$.
Then there is $z\in K$ such that $T^{n_k}(x_1,z)\lra (x_1(i),x'_2(i))\in M\cap (U_1\times U_2)$ by
Lemma~ \ref{surjection} using the distality of $p$.

As $(x_1(i),x_2(i))$ is proximal, $(x_1(i),x'_2(i))$ (in the orbit closure of $(x_1,x_2)$) is distal
and $x_2(i),x'_2(i)\in U_2$ we know that
$$\{n\in\Z_+:d(T^nx_1(i), T^nx_2(i))<\delta/6\}\in \F_t.$$ By the definition of $\delta$ we get
$\inf_{k\in\Z_+} d(T^kx_1(i),T^kx'_2(i))\ge \delta$ which implies that
$$\{n\in\Z_+:d(T^nx_2(i), T^nx'_2(i))>\delta/6\}\in \F_t.$$
Since this holds for each neighborhood of $x_2$, we conclude that $X_3$ is strongly $\F_t$-sensitive.

%At this moment we can show the case when $\pi_1$ is almost 1-1 and $\pi_2$ is $n$ to 1 for some $n\ge 2.$
%Let $\pi=\pi_1\circ \pi_2: X_3\lra X_1$. Since $\pi_1$ is almost 1-1, there is $x\in X_1$ such that $\pi_1^{-1}(x)=\{x'\}$ and
%$\pi_2^{-1}(x')=\{x_1,\ldots,x_n\}$.
%As $(x_1,\ldots,x_n)$ is a minimal point, we may assume that
%$$\delta=\min_{1\le i<j\le n}\{\inf_{n\in\Z_+} d(T^nx_i,T^nx_j)\}>0.$$

%Let $U_i$ be open neighborhood of $x_i$ with $\diam(U_i)<\delta/6$, $1\le i\le n$. As $P(X_3)$ is not closed we
%may assume that $(x_1(i),x_2(i))\lra (x_1,x_2)$
%with $(x_1(i),x_2(i))$ proximal. It is clear that $\pi(x_1(i))=\pi(x_2(i))$ and we may
%assume that $(x_1(i),x_2(i))\in U_1\times U_2$ for all $i\in\N$ .
%Assume that $T^{n_k(i)}(x_1)\lra x_1(i)$. It is clear when $i$ larger there is a subsequence $n'_k(i)$ of $n_k(i)$ such that
%$$T^{n'_k(i)}(x_1,x_2,\ldots,x_n)\lra (x_1(i),x'_{j_2}(i),\ldots,x'_{j_n}(i))\in U_1\times U_{j_2}\times \ldots \times U_{j_n},$$
%where $(j_2,\ldots,j_n)$ is a permutation of $\{2,\ldots,n\}$ since $\pi^{-1}(x)=\{x_1,\ldots,x_n\}$.
%Let $j_l=2$, then $T^{n'_k(i)}(x_1,x_j)\lra (x_1(i), x'_{2}(i))\in U_1\times U_2.$ As $(x_1(i),x_2(i))$
%proximal, $(x_1(i),x'_2(i))$ (in the orbit closure of $(x_1,x_j)$) is distal
%and $x_2(i),x'_2(i)\in U_2$ we know that
%$$\{n\in\Z_+:d(T^nx_1(i), T^nx_2(i))\}<\delta/6\}\in \F_t.$$ By the definition of $\delta$ we get
%$\inf_{k\in\Z_+} d(T^kx_1(i),T^kx'_2(i))\ge \delta$ which implies that
%$$\{n\in\Z_+:d(T^nx_2(i), T^nx'_2(i))\}>\delta/6\}\in \F_t.$$

%Here we need to use the openness of $\pi_2$ and the distal of $pi_2$.

\end{proof}

\begin{lem}\label{5001}Let $(Z_{n+1},T)$ be minimal and consider the strictly PI tower
$ Z_1\overset{\theta_1}{\leftarrow}Y_1\overset{\rho_2}{\leftarrow}Z_2 \overset{\theta_2}{\leftarrow}
Y_2\overset{\rho_3}{\leftarrow}Z_3\overset{\theta_3}{\leftarrow}  \ldots \overset{\rho_n}{\leftarrow}
Z_n\overset{\theta_n}{\leftarrow}Y_n\overset{\rho_{n+1}}{\leftarrow}Z_{n+1}$, where $\theta_i$ is a non-trivial proximal
extension, $\rho_i$ is a non-trivial distal extension and $Z_1$ is distal.
If $(Z_{n+1},T)$ is not strongly $\F_t$-sensitive, then $P(Z_{n+1})$ is closed.
\end{lem}
\begin{proof}We prove this lemma by induction on $n$. For $n=1$ it is just Theorem \ref{very-i}.
Now we assume that the theorem holds for $n\leq k-1$, we prove it still hold for $n=k$.
Let $Z_D$ be the maximal distal factor of $Z_2$ and let $\pi_1: Z_2\rightarrow Z_D$ be the factor map.
Since $P(Z_{2})$ is closed, $\pi_1$ is proximal. So $\pi_2=\pi_1\theta_2: Y_2\rightarrow Z_D$ is proximal.
Consider the new PI tower $$ Z_D \overset{\pi_2}{\leftarrow}
Y_2\overset{\rho_3}{\leftarrow}Z_3\overset{\theta_3}{\leftarrow}  \ldots \overset{\rho_k}{\leftarrow}
Z_n\overset{\theta_k}{\leftarrow}Y_k\overset{\rho_{k+1}}{\leftarrow}Z_{k+1}$$ and the theorem holds for $n\leq k-1$,
we know that $P(Z_{k+1})$ is closed, i.e. the theorem holds for $n=k$.
\end{proof}

We also need the following two  lemmas for the proof of Theorem D.

\begin{lem}\label{5002}Let $\pi: X\rightarrow Y$ be a factor map between minimal systems.
\begin{enumerate}
\item  If $P(X)$ is closed, then $P(Y)$ is closed.
\item  If $\pi$ is proximal and $P(Y)$ is closed, then $P(X)$ is closed
\end{enumerate}
\end{lem}
\begin{proof}(1) follows from Lemma 2 in \cite{AH}.

(2) Let $(x_i,x_i')$ be proximal pairs in $P(X)$ such that $(x_i,x_i')\rightarrow (x,x')$.
Then $(\pi(x_i),\pi(x_i'))$ are proximal pairs in $P(Y)$ such that $(\pi(x_i),\pi(x_i'))\rightarrow (\pi(x),\pi(x'))$.
Since $P(Y)$ is closed, $(\pi(x),\pi(x'))\in P(Y)$. So there exists $p\in E(X)$ (where $E(X)$ is the Ellis semigroup of $X$)
 such that $p\pi(x)=p\pi(x')$, i.e. $\pi(px)=\pi(px')$. Since $\pi$ is proximal, there exists $q\in E(X)$ such that $qpx=qpx'$, i.e.,
$(x,x')\in P(X)$. So $P(X)$ is closed.
\end{proof}

\begin{lem}\label{5003}Let $X$ be an inverse limit of minimal systems $\{(X_i,T_i)\}_{i=1}^{\infty}$.
If $P(X_i)$ is closed for each $i\in\N$, then $P(X)$ is closed.
\end{lem}
\begin{proof} If $P(X_i)$ is closed, then by Theorem 2 in \cite{Auslander1960}, $P(\Pi_{i=1}^{\infty} X_i)$ is closed.
By the definition of inverse limit, $P(X)$ is closed.
\end{proof}

With the above preparation we are ready to give the proof.

\begin{proof}[Proof of Theorem D]
(2) $\Rightarrow$ (1) follows from Proposition~\ref{961-2}.

(1) $\Rightarrow$ (2) We assume that $(X,T)$ is invertible first.

By Proposition \ref{1009-4} $(X,T)$ is PI. Consider the
strictly PI-tower in the structure theorem,
$$Z_1\overset{\theta_1}{\leftarrow}Y_1\overset{\rho_2}{\leftarrow}Z_2 \overset{\theta_2}{\leftarrow}
Y_2\overset{\rho_3}{\leftarrow}Z_3\overset{\theta_3}{\leftarrow}  \ldots X_\infty.$$
By Proposition \ref{1009-2} $X_\infty$ is not strongly $\F_t$-sensitive.
So each finite tower
$$ Z_1\overset{\theta_1}{\leftarrow}Y_1\overset{\rho_2}{\leftarrow}Z_2 \overset{\theta_2}{\leftarrow}
Y_2\overset{\rho_3}{\leftarrow}Z_3\overset{\theta_3}{\leftarrow}  \ldots \overset{\rho_n}{\leftarrow}
Z_n\overset{\theta_n}{\leftarrow}Y_n\overset{\rho_{n+1}}{\leftarrow}Z_{n+1}$$
is not strongly $\F_t$-sensitive.

\medskip

Then By Lemma \ref{5001}, $P(Z_{n+1})$ is closed.
So $P(Y_{n})$ is closed by Lemma \ref{5002}. By Lemma \ref{5003}, $P(X_{\infty})$ is closed.
By Lemma \ref{5002}, $P(X)$ is closed. So $P(X)$ is an equivalence relation~\cite{Sha},
then $\pi\colon X\to X_D$ is proximal.

When $(X,T)$ is not invertible,
let $(\widetilde{X},\widetilde{T})$ be the natural extension of $(X,T)$.
If $P(X,T)$ is not closed, then by Lemma \ref{5002} $P(\widetilde{X},\widetilde{T})$ is not closed.
Since $(\widetilde{X},\widetilde{T})$ is an invertible minimal system, $(\widetilde{X},\widetilde{T})$ is
strong $\F_{t}$-sensitive. So by Proposition \ref{natral-1}, $(X,T)$ is strong
$\F_{t}$-sensitive, a contradiction. So $P(X,T)$ is closed, then $\pi\colon X\to X_D$ is proximal.

\end{proof}

To get a better understanding of Theorem \ref{2-1-1}, we give a well know example which is strongly $\F_t$-sensitive.

To do so, first we give some other criteria of strongly $\F_t$-sensitivity.
\begin{prop}\label{exam00} Let $(X,T)$ be minimal and invertible. If there are $x\not=y$ such that $x,y$
is proximal for $T^{-1}$ and $\inf_{n\in \Z_+}d(T^nx,T^ny)>0$, then $(X,T)$ is strongly $\F_t$-sensitive.
\end{prop}
\begin{proof} Let $\inf_{n\in \Z_+}d(T^nx,T^ny)=2\delta>0$, $U$ be any open set of $X$ and $l\in \N$ with $T^lx\in U$.
Put $x_1=T^lx$ and $y_1=T^ly$, then $(x_1,y_1)$ is proximal for $T^{-1}$ and $\inf_{n\in \Z_+}d(T^nx_1,T^ny_1)\geq 2\delta$.
Since $U$ is a neighborhood of $x_1$, there is $\ep>0$ such that $B_\ep(x_1)\subset U$ for $\ep<\frac{\delta}{10}$.
Set $V=B_{\frac{\ep}{2}}(x_1)$.
Since $x_1,y_1$ is proximal for $T^{-1}$, $\{n<0: d(T^nx_1,T^ny_1)<\ep/2\}$ is thick in $\Z_-$. As $(X,T)$ is minimal,
we know that $(X,T^{-1})$ is minimal. Thus,
$\{n<0: T^nx_1\in V\}$ is syndetic in $\Z_-$. There is $s<0$ such that $T^sx_1\in V$ and $d(T^sx_1,T^sy_1)<\ep/2$.
This implies that $T^{s}x_1,T^{s}y_1\in U$. Set $z_1=T^{s}x_1$ and $z_2=T^{s}y_1$. Then
$\{m\in\Z_+:d(T^mz_1,T^mz_2)>\delta\}=[-s,\infty)$ is thick.
So $(X,T)$ is strongly $\F_t$-sensitive.
\end{proof}

We will give an application of Proposition \ref{exam00}, namely we shall show that the Morse minimal system is strongly $\F_t$-sensitive.
The following results related to Morse system are basic and well known, see for example \cite{g-book}.

The Morse sequence $\omega(n)$:
$$0110100110010110\cdots $$
 can be described by the following algorithms.

$\omega(0)=0, \omega(2n)=\omega(n), \omega(2n+1)=1-\omega(n) (n\in \N).$
Considering $\omega$ as an element of $\Omega=\{0,1\}^{\Z}$ where $\omega(-n)=\omega(n-1)$, let $X \subset \Omega$ be
its orbit closure under the shift $\sigma$ with $\sigma\xi(n)=\xi(n+1)$. Then $(X,\sigma)$ is a minimal flow called
the Morse minimal set.

The homeomorphism $\varphi:\xi \ra \overline{\xi}$
 where $\overline{\xi}(n)=\overline{\xi(n)}$ (and $\overline{0}=1, \overline{1}= 0$) preserves X and commutes
with $\sigma$. The quotient space $Y$, of $X$ modulo the group $\{\varphi, \varphi^2=id\}$
is a factor of $(X,\sigma)$  in the sense that the natural projection $\pi_1: X\ra Y$ satisfies $\pi_1\sigma=\sigma\pi_1$.
For every $\xi \in X$ there exists a sequence $k_i$ such that $\sigma^{k_i}\ra \xi$
and we can associate with $\xi$ the dyadic sequence $\{a_n\}$, $0 \leq a_n \leq 2^n - 1$,
according to the rule $a_n = \lim\{k_i($mod $2^n)\}$. It is easy to check that this
limit exists and is independent of the particular choice of the sequence $\{k_i\}$.
Clearly also the dyadic sequences corresponding to $\xi$ and $\overline{\xi}$ coincide,
so that we can consider the map $\pi_2:Y\ra G$ where $G$ is the compact
group of sequences $\{\{a_n\}:0 \leq a_n \leq 2^n-1, a_{n+1}=a_n($mod $2^n)\}$. Moreover $\pi_2\sigma y=(\pi_2y)+1$
 where $1 = (1,0,0,\ldots,)\in G$. In fact it is not hard to
describe  $\pi_2$ explicitly. If $\eta \in \Omega$  is defined by $\eta(n)=\omega(n)$ for $n\geq 0$ and
$\eta(n)=\overline{\omega(n)}$ for $n< 0$ then $\eta \in X$  and denoting $y_1=\pi_1(\omega), y_2=\pi_1(\eta)$
we have for all $n\in \Z$, $\pi_2^{-1}(n\cdot 1)=\{\sigma^ny_1, \sigma^ny_2\}$ while $\pi_2^{-1}(g)$ is a singleton
for every $g\in G\setminus \{n\cdot 1;n\in \Z\}$. The map $\pi_2$ is therefore almost one to one hence proximal.

\begin{exmp} The Morse minimal system is strongly $\F_t$-sensitive.
\end{exmp}
\begin{proof} Let $X$ be the Morse minimal system. Then $\pi_1:X\ra Y$ is a
group extension and $\pi_2: Y\ra G$ is an
almost one-to-one extension. It is easy to see that
$\inf_{n\in \Z_+}d(\sigma^n \overline{\omega},\sigma^n\eta)>0$ and $(\overline{\omega},\eta)$ is asymptotic for $\sigma^{-1}$.
By Proposition \ref{exam00}, we conclude that the Morse minimal system is strongly $\F_t$-sensitive.

\end{proof}

\subsection{Strong \texorpdfstring{$\F_{Poin_d}$}{FPoind}
and \texorpdfstring{$\F^*_{d,0}$}{F*d0}-sensitivity}\label{sec-5.2}
In this subsection, we discuss strong $\F_{Poin_d}$ and $\F^*_{d,0}$-sensitivity. In this subsection we
assume that $T$ is a homeomorphism.

\begin{defn}
Let $(X,T)$ be a t.d.s. We say $(X,T)$ is strongly $\F_{Poin_d}$ (resp. $\F^*_{d,0})$-sensitive if there is $\delta>0$
such that for each opene subset $U$ of $X$, there are $x,y\in U$ with
$\{n\in\Z:d(T^nx,T^ny)>\delta\} \in \F_{Poin_d}$ (resp. $\F^*_{d,0})$
\end{defn}

We state some basic notations, definitions and results related to $\F_{Poin_d}(resp. \F^*_{d,0})$ first.

We say that $S \subset \Z$ is a set of {\em $d$-recurrence} if
for every measure preserving system $(X,\chi,\mu,T)$ and for every
$A\in \chi$ with $\mu (A)> 0$, there exists $n \in S\setminus\{0\}$  such that
$\mu(A\cap T^{-n}A\cap \ldots \cap T^{-dn}A)>0.$
Let $\F_{Poin_d}$  \index{$\F_{Poin_d}$} be
the family consisting of all sets of $d$-recurrence. By Furstenberg's multiple ergodic theorem
the definition is reasonable.
%It is known that each IP-set is in  $\F_{Poin_d}$.
A striking result due to Furstenberg and Katznelsen \cite[Theorem C]{F-K} in our terms
is that $\F_{fip}\subset \F_{Poin_d}$. So we have

\begin{prop} If a minimal system $(X,T)$ is not strongly $\F_{Poin_d}$-sensitive, then
it is an almost one-to-one extension of its maximal $\infty$-step nilfactor.
\end{prop}
\begin{proof} It follows from the fact $\F_{fip}\subset \F_{Poin_d}$ and Theorem B.
\end{proof}

A subset $A$ of $\Z$ is a {\em Nil$_d$ Bohr$_0$-set}
\index{Nil$_d$ Bohr$_0$-set} if there exist a $d$-step nilsystem
$(X,T)$, $x_0\in X$ and an open neighborhood $U$ of $x_0$ such that
$N(x_0,U)=:\{n\in \Z: T^n x_0\in U\}$ is contained in $A$. Denote by
$\F_{d,0}$ \index{$\F_{d,0}$} the family consisting of all
Nil$_d$ Bohr$_0$-sets. Let $\F_{GP_d}$  be the family generated by the sets of forms
$$\bigcap_{i=1}^k\{n\in\Z:P_i(n)\ (\text{mod}\ \Z)\in (-\ep_i,\ep_i)
\},$$ where $k\in\N$, $P_1,\ldots,P_k$ are generalized polynomials
of degree $\le d$, and
$\ep_i>0$. For the definition of generalized polynomials, see~\cite[Page 21]{WSY2015}.
We have \cite[Proposition 7.21, Proposition 7.24]{WSY2015}
for each $d \in \N$, $\F_{d,0}$ is a filter, and $\F_{Poin_d}$ has the Ramsey
property.

The following two lemmas will be used in the next theorem.

\begin{lem}\cite[Theorem E]{WSY2015}\label{RPd1} Let $(X,T)$ be a minimal system and
$x,y\in X$. Then the following statements are equivalent for
$d\in\N\cup\{\infty\}$:

\begin{enumerate}
\item $(x,y)\in \RP^{[d]}$.
\item $N(x,U)\in \F_{d,0}^*$ for each
neighborhood $U$ of $y$.
\item $N(x,U)\in \F_{Poin_d}$ for each
neighborhood $U$ of $y$.

\end{enumerate}
\end{lem}

\begin{lem}\cite[Theorem F]{WSY2015}\label{RPd2} Let $(X,T)$ be a minimal system,
$x\in X$ and $d\in\N\cup\{\infty\}$. Then the following statements
are equivalent:
\begin{enumerate}
\item $x$ is a $d$-step AA point.

\item $N(x,V)\in \F_{d,0}$ for each neighborhood $V$ of $x$.

\item $N(x,V)\in \F_{Poin_d}^*$ for each
neighborhood $V$ of $x$.

\end{enumerate}
\end{lem}

Using Lemma \ref{RPd2} instead of using Proposition \ref{921-1} we have the following result by the same proof of Theorem \ref{sensitive}.

\begin{thm} Let $(X,T)$ be a minimal system and $d\in\N$. Then $(X,T)$ is $\F_{Poin_d}$-sensitive if and only if
$\pi:X\lra X_{eq}$ is not almost one-to-one.
\end{thm}

Using the idea of the proof of Proposition \ref{925-1} and
Proposition~\ref{925-4}, we obtain the following result.

\begin{thm}\label{1011-1}
If $(X,T)$ is a strongly $\F_{Poin_d}$-sensitive minimal system, then $\pi: X\lra X_{d}$ is not an almost one-to-one extension.
\end{thm}
\begin{proof}
Suppose that $(X,T)$ is strongly $\F_{Poin_d}$-sensitive with the sensitive constant
$10\delta$ and $\pi: X\lra X_{d}$ is an almost one-to-one extension.
Then there is $x\in X$ such that  $\RP^{[d]}[x]=x$.

Let $\delta'<\delta$ and $U=B(x,\delta')$, then there are $y,z\in U$ such that
$$F=\{n\in \Z:d(T^ny,T^nz)>10\delta\}\in \F_{Poin_d}.$$  By the Ramsey property of $\F_{Poin_d}$,
$F_1=\{n\in \Z:d(T^nx,T^nu)>5\delta\}\in \F_{Poin_d},$ where  $u=y$ or $u=z$.
As $\RP^{[d]}[x]=x$, by Lemma \ref{RPd1} we have
$F_2=N(x,U)\in \F_{Poin_d}^*$. So
$$F_3=F_1\cap F_2\subset \{n\in \Z:d(x,T^nu)>5\delta-\delta'\}\in \F_{Poin_d}.$$

Then by the Ramsey property of $\F_{Poin_d}$ and
using the same argument as in the proof of Proposition~\ref{925-4},
we deduce that there are $v\in X$ with $d(u,v)\geq  \delta$ and for each neighborhood
$V$ of $v$, $N(u,V)\in \F_{Poin_d}.$ It is clear that $(u,v)\in \RP^{[d]}(X)$ by Lemma \ref{RPd1}. Moreover,
we know that $\pi(u)=\pi(v)$ since $\RP^{[d]}(X_d)=\Delta$. This contradicts to the fact that $\RP^{[d]}[x]=x$,
showing that $\pi$ is not almost one-to-one.

% Then we claim: there is
% $v\in X$ such that $(u,v)\in \RP^{[d]}$ and $d(u,v)\geq  \frac{\delta}{20}$.
% Suppose not, $\RP^{[d]}[u]\subset B(u,\frac{\delta}{20})\subset V'$,
% then for any $t\in X\setminus V'$, we have an neighborhood $V_t$ of $t$ with $V_t\subset X\setminus V$ such that
% $N(u,V_t)\not \in  \F_{Poin_d}$. As $X\setminus V'$ is compact, we can cover $X\setminus V'$ by finite set
%$ \{V_{t_1},V_{t_2},\cdots ,V_{t_k} \} $. Let $W=\underset{i=1}{\overset{k}{\cup}}V_{t_i}$, then
%$X\setminus V' \subset W \subset X\setminus V$, and $N(u,W)\not \in  \F_{Poin_d}$. But
%$F_3=F_1\cap F_2\subset \{n\in \Z_+:d(x,T^nu)>\frac{\delta}{2}-\delta'\}$ is a  $\F_{Poin_d}$ set, a contradiction.
%So  we have $v\in X,(u,v)\in \RP^{[d]}$ and $d(u,v)\geq  \frac{\delta}{20}$.

%So we have $\underset{k\rightarrow \infty}{\lim}x_k= x$ and $y_k\in X$ such that $(x_k,y_k)\in \RP^{[d]}$ and
%$(x_k,y_k)\geq \frac{\delta}{20}$. Let $\pi(x_k)=z_k$ and define a function $D(z)=diam(\pi^{-1}(z))$. It is easy to
%see that $D$ is a uppercontinuous function on $X_{d}$. Then $\frac{\delta}{20}\leq\underset{k\rightarrow
%\infty}{\limsup}(D(z_k))\leq D(z_0)=0$, which leads a contradiction.
\end{proof}

\begin{cor}
If $(X,T)$ is strongly $\F_{d,0}^{*}$-sensitive minimal system, then $\pi: X\lra X_{d}$ is not an one-to-one extension.
\end{cor}
\begin{proof}
The proof is similar with Theorem \ref{1011-1}.
\end{proof}

%\begin{thm}
%If $(X,T)$ is not strongly $\F_{Poin_d}$ sensitive, then $\pi: X\lra Z_{\infty}$ is  AA.
%\end{thm}
%\begin{proof}
%By theorem, we know $\F_{ip}\subset \F_{Poin_d}$. As  $(X,T)$ is not strongly $\F_{Poin_d}$ sensitive,
%$(X,T)$ is not strongly $\F_{ip}$ sensitive, by theorem  $\pi: X\lra Z_{\infty}$ is  AA.
%\end{proof}

%If $(X,T)$ is $\pi: X\lra Z_{d}$ is not AA, $(X,T)$ may not be strongly $\F_{Poin_d}$ sensitive.
It is unexpected that the converse of Theorem \ref{1011-1} fails.  To give a counter-example we need

\begin{lem}\label{april-1}\cite[Theorem B, Corollary D]{WSY2015} For $d\in\N$, $\F_{d,0}=\F_{GP_d}$ and $\F_{Poin_d}\subset \F^*_{d,0}.$
\end{lem}

\begin{exmp}\label{last-e} There is a minimal system which is not an almost one-to-one
extension of the maximal $(d-1)$-step nilfactor and the system is not strongly $\F_{Poin_{d-1}}$-sensitive.
\end{exmp}
\begin{proof} For $d\ge 2$ define $T_{\alpha, d}:\T^d \lra \T^d$ by
$$ T_{\alpha, d}(\theta_1,\theta_2,\cdots,\theta_d)=(\theta_1+\alpha,\theta_2+\theta_1,\cdots,\theta_d+\theta_{d-1})$$
where $\alpha \in \R$.
When $\alpha \in \R \setminus \mathbb{Q}$, $(\T^d, T_{\alpha, d})$ is minimal.
A simple computation yields that
$$ T_{\alpha,d}^n(\theta_1,\theta_2,\cdots,\theta_d)=(\theta_1+n\alpha,\theta_2+n\theta_1+\frac{1}{2}n(n-1)\alpha,\cdots,
\underset{i=0}{\overset{d}{\Sigma}}\tbinom{n}{d-i}\theta_i) $$
where $\theta_0=\alpha, n\in \Z$ and $\tbinom{n}{0}=1, \tbinom{n}{i}=\frac{\prod_{j=0}^{i-1}(n-j)}{i!}$ for $i=1,2,\cdots,d$.

$(\T^d, T_{\alpha, d})$ is a $d$-step nilsystem, so we have $\RP^{[d]}(\T^d)=\Delta_{T^d}$, and for $s<d$
$$\RP^{[s]}(\T^d)=\{({\bf x},{\bf y}): \text{the first}\ s\ \text{coordinates of}\ {\bf x},{\bf y}\ \text{are the same}\}.$$
When $\alpha \in \R \setminus \mathbb{Q}$, $(\T^d, T_{\alpha, d})$ is minimal and not an almost one-to-one extension of its maximal $(d-1)$-step nilfactor.
We will prove that $\T^d$ is not strongly $\F_{Poin_{d-1}}$-sensitive.

Assume the contrary that it is strongly $\F_{Poin_{d-1}}$-sensitive.
%Then $\T^d$ is strong  $\F_{Poi_{d-1}}$ sensitive.
That is, there is $\delta>0$ such that
for any ${\bf x}\in \T^d$ and $\ep\in \R$, there is ${\bf y}\in \T^d$ such that $\|{\bf {x-y}}\|<\ep$
and $\{n\in \Z: d(T^n_{\alpha,d}{\bf x},T^n_{\alpha,d}{\bf y})>2\delta\}\in \F_{Poin_{d-1}}$.
We can choose ${\bf x}={\bf 0}$ and $\varepsilon=\delta$, then
we have ${\bf y}=(y_1,y_2,\cdots,y_d)$ with
$$\{n\in\Z: d(T^n_{\alpha,d}{\bf 0},T^n_{\alpha,d}{\bf y})>2\delta\}\in \F_{Poin_{d-1}}$$
and $\|{\bf y}\|<\delta$. A simple computation yields that
$$T^n_{\alpha,d}{\bf y}-T^n_{\alpha,d}{\bf 0}= (y_1,0,\ldots,0)+(0,T_{y_1, d-1}^n(y_2,y_3,\cdots,y_{d}))$$
So we have $$F_1=\{n\in\Z:\|T_{y_1, d-1}^n(y_2,y_3,\cdots,y_{d})\| \geq\delta\}\in \F_{Poin_{d-1}}$$
since $F_1\supset \{n\in\Z: d(T^n_{\alpha,d}{\bf 0},T^n_{\alpha,d}{\bf y})>2\delta\}.$

Define
\begin{align*}
F_2= \{n\in\Z:\text{the absolute value of each coordinate of }   \\
T_{y_1, d-1}^n(y_2,y_3,\cdots,y_{d})\ \text{is less than}\tfrac{\delta}{(d-1)} \}.
\end{align*}
We know that $F_2\in \F_{GP_{d-1}}$ by the definition of generalized polynomials. Moreover, we have
$$F_2\subset F_3=\{n\in\Z:\|T_{y_1, d-1}^n(y_2,y_3,\cdots,y_{d})\| <\delta\}.$$

Thus, $F_1^c=F_3\supset F_2$ which implies that $F_1^c\in \F_{GP_{d-1}}=\F_{d-1,0}.$  So $F_1\not \in \F_{d-1,0}^*$ which
implies that $F_1\not \in \F_{Poin_{d-1}}$ by Lemma \ref{april-1}, a contradiction.

\end{proof}

\end{document}